\documentclass[11pt]{amsart}

\usepackage{amssymb,amsmath,amsthm,amsfonts,graphicx,hyperref,url,dsfont}

\usepackage{listings}
\usepackage{xcolor}
\usepackage[figure]{hypcap}

\usepackage{siunitx}
\usepackage{booktabs}
\usepackage[english]{babel}
\usepackage{blindtext}
\RequirePackage{fix-cm}

\newtheorem{theorem}{Theorem}

\theoremstyle{definition}
\newtheorem{definition}{Definition}

\theoremstyle{remark}
\newtheorem{remark}{Remark}

\numberwithin{example}{section}
\numberwithin{remark}{section}
\numberwithin{remarks}{section}
\numberwithin{theorem}{section}
\numberwithin{equation}{section}
\numberwithin{definition}{section}
\numberwithin{algorithm}{section}
\numberwithin{assumption}{section}
\numberwithin{table}{section}
\numberwithin{figure}{section}

\DeclareMathOperator{\Var}{Var}

\newcommand{\cL}{\mathcal{L}}

\newcommand{\cD}{\mathcal{D}}
\usepackage[T1]{fontenc}
\usepackage[utf8]{inputenc}
\usepackage{combelow}

\usepackage{newunicodechar}
\newunicodechar{ș}{\cb{s}}

\usepackage{blindtext}
\usepackage{float}

\usepackage{pgfplots,rotating}
\pgfplotscreateplotcyclelist{short_list}{%
	mblue,line width=1.5pt, mark=none, solid\\
	mgreen,line width=1.5pt, mark=none, solid\\
	morange,line width=1.5pt, mark=none, solid\\
	brown,line width=1.5pt, mark=none, solid\\
	gray, line width=1pt, mark=none, dashed\\%
	gray, line width=1pt, mark=none, dotted\\%
	gray, line width=1pt, mark=none, dashdotted\\%
} 
\pgfplotsset{height=0.4\linewidth,width=0.96\linewidth,compat=1.9,
	every axis/.append style={legend style={
			/tikz/every even column/.append style={column sep=9pt}}}}

\usepackage{pgfplots, pgfplotstable}

\usepackage{changes}
\usepackage{etoolbox}
\makeatletter
\patchcmd{\@addmarginpar}{\ifodd\c@page}{\ifodd\c@page\@tempcnta\m@ne}{}{}
\makeatother
\reversemarginpar
\definecolor{forestgreen}{rgb}{0.13, 0.55, 0.13}
\usepackage{changes}
\usepackage[color=yellow!50!white, textwidth=5.0cm]{todonotes} 
\usepackage{etoolbox}
\makeatletter
\patchcmd{\@addmarginpar}{\ifodd\c@page}{\ifodd\c@page\@tempcnta\m@ne}{}{}
\makeatother
\reversemarginpar

\begin{document}
\title[A splitting/polynomial chaos expansion approach]{A splitting/polynomial chaos expansion approach for stochastic evolution equations}

\author[A. Kofler, T. Levajkovi\'c, H. Mena, A. Ostermann]{Andreas Kofler$^1$, Tijana Levajkovi\'c$^2$, \\ Hermann Mena$^3$ and Alexander Ostermann$^4$}

\thanks{$^1$Andreas KOFLER, {Department of Radiology, Charit\'e-Universit\"atsmedizin Berlin, Germany, \texttt{andreas.kofler@charite.de}}\\
	\indent$^2$Tijana LEVAJKOVI\'C, {Institute of Stochastics and Business Mathematics, Vienna University of Technology,  Austria, \texttt{tijana.levajkovic@tuwien.ac.at}}\\
\indent$^3$Hermann MENA, Department of Mathematics, Yachay Tech University,  Urcuqu\'i, Ecuador,  \texttt{mena@yachaytech.edu.ec} \& Department of Mathematics, 
University of Innsbruck,  Austria, \texttt{hermann.mena@uibk.ac.at}\\
\indent$^4$Alexander OSTERMANN, Department of Mathematics, 
University of Innsbruck,  \linebreak Austria, \texttt{alexander.ostermann@uibk.ac.at}}


\maketitle
\begin{abstract}
In this paper we combine deterministic splitting methods with a polynomial chaos expansion method for solving stochastic parabolic evolution problems. The stochastic differential equation is reduced to a system of deterministic equations that we solve explicitly by splitting methods.  The method can be applied to a wide class of problems  where the related stochastic processes are given uniquely in terms of stochastic polynomials.  A 
comprehensive convergence analysis is provided and numerical experiments validate our approach.
\end{abstract}

\section{Introduction} 
\label{intro}
Splitting methods are  numerical methods for solving differential equations, both ordinary and partial differential equations (PDEs), involving operators that are decomposable into a sum of (differential) operators. These methods  are used to improve the speed of calculations for problems involving decomposable operators and to solve multidimensional PDEs by reducing them to a sum of  one-dimensional problems \cite{FaouOS15}. Splitting methods have been successfully applied to many types of PDEs, e.g. \cite{HanO09}, \cite{HanOII09}. Exponential splitting methods are applied in  cases when the explicit solution of a splitted equation  can be computed. Such computations often rely on applying fast Fourier techniques, see for instant \cite{Trefethen}. 
Resolvent splitting is used in cases when the splitted equation cannot be solved explicitly \cite{resolvent}, \cite{Ostermann12}; here we consider  this type of methods. 

There are also many results in the literature about the  approximation of solutions of SPDEs using splitting methods, see e.g. \cite{BarbuRoeckner2017}, \cite{Bensoussan92}, \cite{CHP_2012}, \cite{Cohen2017}, \cite{Faou_Stoc}, \cite{GreckshLisei2013}, \cite{GyongyKrylov2003}  and refe\-rences therein. In \cite{GreckshLisei2013} a splitting  method for nonlinear stochastic equations of Schr\"odinger type is proposed. There the authors approximate the solution of the problem by a sequence of solutions of two types of equations: one without stochastic  term and other containing only the stochastic term. They prove that 
an appropriate combination of the 
solutions of these equations converges strongly to the solution of the original problem.  Exponential integrators for nonlinear Schr\"odinger equations with white noise dispersion were proposed in \cite{Cohen2017}. For a stochastic incompressible time-dependent Stokes equation different time-splitting methods were studied in \cite{CHP_2012}. In \cite{BarbuRoeckner2017} the convergence of a Douglas--Rachford type splitting algorithm is presented for general SPDEs driven by linear multiplicative noise. In this work a splitting/polynomial chaos expansion is considered for stochastic evolution equations. Our approach has not been considered in the literature  for solving these types of SPDEs so far. 

We consider stochastic evolution equations  of the form 
\begin{equation}
\label{evolution eq}
\begin{split}
du(t) &= \big( (A+B )\, u(t) +f(t)\big) \, dt + \big(C \, u(t) + g(t)\big) \, dB(t)\\
u(0) & = u^0, 
\end{split}
\end{equation} 
where  $A$, $B$ and $C$ are differential operators acting on Hilbert space valued stochastic processes, $\{B_t\}_{t\geq 0}$ is a cylindrical Brownian motion on a given probability space $(\Omega, \mathcal F, \mathbb P)$ and $f$ and $g$ are deterministic functions.  
In \cite{LotRoz} equation \eqref{evolution eq} involving Gaussian noise terms was solved in an  appropriate weighted Wiener chaos space.   The deterministic problem that corresponds  to \eqref{evolution eq}, i.e.,  the case where $C=0$ and $g=0$,    
for particular  $Au=\partial_x(a\partial_x u)$, $Bu=\partial_y(b\partial_y u)$ and $f$ was studied in  \cite{FaouOS15}.  
We  consider equation \eqref{evolution eq} \linebreak  involving a non-Gaussian noise term. 
Namely, we consider  inhomogeneous parabolic evolution equations   involving the operators that can be split in   $A+B$ and  uniformly distributed random inputs.    These equations, can be also written in the form 
\begin{equation}
\label{Eq-1.1}
\begin{split}
\phantom{(0)} u_t (t, x, \omega) &= (A+B) \, u(t, x, \omega) + G(t, x, \omega) \\
u(0, x, \omega) &= u^0(x, \omega),
\end{split}
\end{equation} 
where  $G$ represents the noise term, see  e.g. \cite{HOUZ}, \cite{LevM17B}, \cite{LevPSZ15}, \cite{LevPSZ18}, \cite{LotRoz}. The existence of a random parameter $\omega$ is due to uncertainties coming from initial conditions and/or a random force term.  
Therefore, the solution is  considered to be a stochastic process.  

Stochastic processes with finite second moments on white noise spaces can be represented in series expansion form  in terms of a family of orthogonal stochastic polynomials. The classes of orthogonal  polynomials are chosen  depending on the underlying probability  measure  \cite{Hida}, \cite{HOUZ}. 
Namely, the Askey scheme of   hypergeometric ortho\-gonal polynomials and the Sheffer system \cite{Schout}, \cite{Szego} can be used to define several discrete and continuous distribution types   \cite{XiuKarn}.  For example, in the case of the Gaussian measure, the orthogonal basis of the space of random variables with finite second moments is constructed by the use of the Hermite polynomials.  
We consider  \linebreak  problems  
with  non-Gaussian random inputs. The noise term is considered to be  uniformly  distributed.   
It is known that in order to obtain a square integrable solution   of \eqref{evolution eq} with deterministic initial condition,  it is enough to assume that the operator  $A-\frac12 CC^*$ is elliptic and that the stochastic part (the noise term) is sufficiently regular, see e.g.   \cite{DaP_Z}. In this  work, the assumptions on the  input data for  problem  \eqref{Eq-1.1} will be set such that  the existence of a square integrable solution is always established. We do not consider  solutions which are generalized stochastic processes  as in \cite{LevPSZ15}, \cite{LotRoz}, since our focus is on numerical treatment. 

Our approach is general enough to be applied to problems with additive noise,  problems involving  multiplicative noise and problems with  convolution-type noise \cite{LevPSZ15}. For instance, with this approach the heat equation with random potential, the heat equation in random (inhomogeneous and anisotropic) media and the Langevin equation can be solved.  
If \eqref{evolution eq}  does have a sufficiently regular solution, this solution can be projected on an orthonormal basis in some Hilbert space, resulting in a system of equations for the corresponding Fourier coefficients. 
Thus, we use the so-called \emph{polynomial chaos method} or the \emph{chaos expansion method} and define the solution of \eqref{evolution eq} as a formal Fourier series with the coefficients computed by solving the corresponding system  of deterministic PDEs \cite{LotRoz}. With this method, the deterministic part of a solution is separated from its random part. Particularly,  in the case of Gaussian noise, the orthonormal  basis of stochastic polynomials involves the Hermite polynomials and in the case when the noise term is uniformly distributed,  the orthonormal basis involves the Legendre polynomials  \cite{Schout}.  
By construction, the solution is strong in the probabilistic sense. It is uniquely determined by the coefficients, free terms, initial condition and the noise term. 
The coefficients in the Fourier series are uniquely  \linebreak  determined by equation \eqref{evolution eq} and are computed by solving (numerically) the corresponding lower-triangular system of deterministic parabolic equations. 
The polynomial chaos method has been successfully applied for solving  \linebreak  general classes of SPDEs.   The list of references  is long, here we mention just a few \cite{HOUZ}, \cite{LevPSZ15}, \cite{MR}, \cite{NR}.  In   \cite{LevM17B}, \cite{LevMP18}, \cite{LevMT16} this approach has been recently applied to the stochastic optimal regulator \mbox{control problem \cite{HaLLMT15}}. 

Practical application of the Wiener polynomial chaos involves  two  truncations,  truncation with respect to the number of the random variables and truncation with respect to the order of the  orthogonal Askey polynomials used (in the particular case considered, the Legendre polynomials), \mbox{see e.g. \cite{HLRZ}}. 

The paper is organized as follows. In  Section 2 we introduce the notation  and basic concepts used in the following sections. In Section 3 we present the splitting/polynomial chaos expansion approach and provide a complete convergence analysis. Finally, in Section 4 we validate our approach with a numerical experiment.    


%
\section{Preliminaries}
In this section we briefly recall  polynomial chaos representations of random variables and stochastic processes. Particular emphasis is given to  Legendre polynomials and the corresponding Wiener--Legendre expansion, and to the Karhunen--Lo\`eve expansion.    

\subsection{Polynomial chaos representation}
Let $\mathcal{I}=(\mathbb{N}_{0}^{\mathbb{N}})_{c}$ be  the set
of sequences of non-negative integers which have only finitely many
nonzero components
$\alpha=(\alpha_1,\alpha_2,\ldots,\alpha_m,0,0, \ldots)$, $\alpha_i\in
\mathbb{N}_0$, $i=1, 2,..., m$, $m\in \mathbb{N}$.  
Particularly, $(0,0, \dots)$ is the zero vector. We denote  by  $\varepsilon^{(k)}=(0,\cdots,0,1,0,\cdots)$, $k\in
\mathbb{N}$ the $k$th unit
vector. The length  of $\alpha\in \mathcal{I}$ is the sum of its components
$|\alpha|=\sum_{k=1}^\infty \alpha_k$.

First, we briefly recall the main results from the Wiener--It\^o chaos expansion. Let  $(\Omega, \mathcal F, \mu)$ be a  probability space with the Gaussian pro\-ba\-bi\-lity measure $\mu$ and let $(L)^2=L^2(\Omega, \mathcal F, \mu)$  denote the space of random variables with finite second moments on the probability space $(\Omega, \mathcal F, \mu)$. The space $(L)^2$ is a Hilbert space. The scalar product of two random variables $F, G\in (L)^2$ is given by 
\[(F(\omega), G(\omega))_{(L)^2} = \mathbb E (F(\omega) \, G(\omega)),  \]where $\mathbb E$ denotes  the expectation with respect to the measure $\mu$. 

Let  $\{h_n\}_{n\in \mathbb N_0}$ be the Hermite polynomials 
given  through  the recursion
\begin{equation*}
\begin{split}
h_0(x) &= 1,\\
h_1(x) &= x,\\
h_{n+1}(x) &=  x h_n(x) + n h_{n-1}(x) \quad \text{for} \,\,\,  n\geq 2, \, x\in \mathbb R. 
\end{split}
\end{equation*} 
Define   the $\alpha$th Fourier--Hermite polynomial as the product 
\begin{equation*}
\begin{split}
H_\alpha({\boldsymbol{\xi}}(\omega)) =  H_{(\alpha_1, \alpha_2, ...)} ((\xi_1(\omega), \xi_2(\omega), ...))  = \prod_{i\in \mathbb N} \, h_{\alpha_i} (\xi_i(\omega)), 
\end{split}
\end{equation*}
represented  in terms of the Hermite polynomials evaluated  at appropriate  components of the sequence $\boldsymbol{\xi}=(\xi_1, \xi_2, ...)$ of  independent Gaussian variables with zero mean and unit variance.  
Especially,  
\[\begin{split}H_{(0,0, \dots)}({\boldsymbol{\xi}}(\omega)) &=  \prod_{i\in \mathbb N } h_0(\xi_i(\omega))  = 1 \quad \text{and}\\
H_{\varepsilon^{(k)}}({\boldsymbol{\xi}}(\omega)) &= h_1(\xi_k(\omega))  \prod_{i \not = k, i\in \mathbb N } h_0(\xi_i(\omega))= \xi_k(\omega), \quad  k\in \mathbb N.
\end{split}\] 

\begin{theorem}[Wiener--It\^o chaos expansion theorem, \cite{HOUZ}] 
	\label{WICE Thm}
	Each square integrable random variable $F\in (L)^2$ can be uniquely represented in the form
	\begin{equation}
	\label{chaos exp -form 1}
	F(\boldsymbol{\xi}(\omega)) = \sum_{\alpha\in \mathcal I} \, f_\alpha \, H_\alpha (\boldsymbol{\xi}(\omega)),
	\end{equation} where $f_\alpha\in \mathbb R$ for $\alpha\in \mathcal I$. Moreover, it holds
	\[\|F\|^2_{(L)^2} = \sum_{\alpha\in \mathcal I} \, f_\alpha^2 \, \,  \|H_\alpha\|^2_{(L)^2} \, < \infty. \]  
\end{theorem}

The family of stochastic polynomials $\{H_\alpha\}_{\alpha\in \mathcal I}$ forms an orthogonal basis of $(L)^2$ such that  
\begin{equation}
\label{H_alpha norms} 
\mathbb E (H_\alpha \, H_\beta) = \alpha! \, \delta_{\alpha \beta}, 
\end{equation}
for all $\alpha, \beta \in \mathcal I$, see \cite{HOUZ}.  Here $\delta_{\alpha\beta}$ denotes the Kronecker delta. 
Thus, the sequence of the coefficients  in \eqref{chaos exp -form 1}, which is  a sequence of real numbers,  is obtained from 
$f_\alpha = \frac{1}{\alpha!} \,  \mathbb E (F \,  H_\alpha)$, $\alpha\in \mathcal I$.  
Also,  we have  
\[\mathbb E(H_{(0,0, \dots)}) =1 \quad \text{and}\quad \mathbb E (H_\alpha)=0 \, \, \, \text{for} \, \,  |\alpha|>0.\]  
Property \eqref{H_alpha norms} is a  consequence of the orthogonality of the Hermite polynomials
\begin{equation*}
\int_{\mathbb R} \, h_n(x) \, h_m(x) \, d\mu(x) = \frac1{\sqrt{2\pi}} \int_{\mathbb R} \, h_n(x) \, h_m(x) \, e^{-\frac{x^2}{2}}  \, dx = n! \, \delta_{m,n} 
\end{equation*}
for all $m,n\in \mathbb N$. 

In \cite{XiuKarn} it was shown that the initial construction of the Wiener chaos which corresponds to the Gaussian measure and  Hermite polynomials    can be extended also to other types of measures, where instead of the Hermite polynomials other classes of orthogonal polynomials from the Askey scheme \cite{Schout} are used. For example, the Gamma distribution corresponds to the Laguerre polynomials and thus to the  Wiener--Laguerre chaos, while the Beta distribution is related to the Jacobi polynomials and thus to the  Wiener--Jacobi chaos etc.  
Moreover, in \cite{Schout} it  was proven that the optimal exponential convergence rate for each Wiener-Askey chaos can be realized. 

In this paper,  we deal with   stochastic evolution problems  with non-Gaussian random inputs which are uniformly distributed.
From the Askey scheme of orthogonal polynomials it follows that 
the  uniform distribution, as a special case of the Beta distribution, corresponds to the special class of the Jacobi polynomials, the Legendre polynomials. 
Therefore, we are going to work with  the Wiener--Legendre polynomial chaos. 

\subsection{Wiener--Legendre chaos representation} 
\label{Sec W-Legendre chaos}
Denote by $\{p_n(x)\}_{n\in \mathbb N_0}$  the Legendre polynomials on $[-1,1]$. These polynomials  are defined by the recursion
\begin{equation}
\label{Legendre polynomials def}
\begin{split}
p_0(x) &= 1,\\
p_1(x) &= x,\\
(n+1)p_{n+1}(x) &=  (2n+1) x p_n(x) - n p_{n-1}(x)  \quad \text{for} \, \, n\geq 1.
\end{split}
\end{equation}They can be also obtained from Rodrigues' formula \cite{Schout}  
\begin{equation*}
p_n(x) = \frac{1}{2^n n!} \, \frac{d^n}{dx^n} \, (x^2-1)^n . 
\end{equation*} 

The Legendre polynomials satisfy the second order  differential equation $(1-x^2) p_n''(x) - 2 x p_n'(x) + n (n+1) p_n(x) = 0$, which appears in physics when solving the Laplace equation in spherical coordinates \cite{Schout}.
These  polynomials are orthogonal and  it holds  
\begin{equation}
\label{orthogonality Legendre pol}
\int_{-1}^1 p_m(x) \, p_n(x) \, dx = \frac{2}{2n+1} \, \delta_{m,n}, \quad m,n\in \mathbb N_0. 
\end{equation}
The previous property \eqref{orthogonality Legendre pol} is equivalent to the orthogonality relation with respect to the uniform measure, i.e.,  the measure with the constant weighting function $w(x)=\frac12$.

We consider square integrable random variables and  stochastic processes  on a probability space $(\Omega, \mathcal F, \mathbb P)$ with the measure $\mathbb P$ generated by the  uniform distribution.    Let $(L)^2= L^2(\Omega, \mathcal F, \mathbb P)$ be the Hilbert space of square integrable random variables with respect to the measure $\mathbb P$.  \\
We define the $\alpha$th Fourier--Legendre polynomial  as the product 
\begin{equation}
\label{Fourier-Legendre pol}
L_\alpha (\boldsymbol{\xi}(\omega)) = \prod_{i\in \mathbb N} \, p_{\alpha_i} (\xi_i(\omega)), \quad \alpha=(\alpha_1, \alpha_2, \dots)\in \mathcal I, 
\end{equation}where $\{p_n\}_{n\in \mathbb N_0}$ are the Legendre polynomials and $\boldsymbol{\xi} = (\xi_1, \xi_2, ...)$ is a sequence of independent uniformly distributed random variables with zero mean and unit variance. Note that the product in \eqref{Fourier-Legendre pol} is finite since each $\alpha\in \mathcal I$ has only finitely many nonzero components. Particularly, 
\[\begin{split}
L_{(0,0, \dots)} (\boldsymbol{\xi}(\omega)) &=1 \quad \text{and}\\  
L_{\varepsilon^{(k)}} (\boldsymbol{\xi}(\omega)) &= \xi_k(\omega) \quad    \text{for} \,   k\in \mathbb{N}.
\end{split}\] 
We also have 
\begin{equation}
\label{expectation Lalpha}
\mathbb E(L_{(0,0, \dots)}) =1 \quad \text{and} \quad \mathbb E (L_\alpha (\boldsymbol{\xi}(\omega)) )=0 \,  \, \, \text{for} \, \,  |\alpha|>0,
\end{equation}
since $\boldsymbol{\xi}(\omega)$ has zero mean.  Moreover, from the orthogonality  \eqref{orthogonality Legendre pol} of the Legendre polynomials we obtain that the family of the Fourier--Legendre polynomials $\{L_\alpha\}_{\alpha\in \mathcal I}$ is also orthogonal and  \begin{equation}
\label{orthogonality Lalfa}
\mathbb E (L_\alpha \, L_\beta) =  \mathbb E L^2_\alpha \, \delta_{\alpha, \beta} = \frac1{\prod_{k\in \mathbb N}{(2\alpha_k+1)}} \,\, \delta_{\alpha \beta} 
\end{equation}for all $\alpha, \beta \in \mathcal I$. 

Now we formulate the representation  of a random variable  in an  analoguous way to Theorem \ref{WICE Thm}.  
\begin{theorem}[Wiener--Legendre chaos  expansion theorem] 
	Each random variable $F\in (L)^2$  can be uniquely  represented in the form
	\begin{equation}\label{FL exp}
	F(\boldsymbol{\xi}(\omega)) = \sum_{\alpha\in \mathcal I} \, f_\alpha \, L_\alpha (\boldsymbol{\xi}(\omega)), 
	\end{equation}where  
	\[f_\alpha=\frac1{\mathbb E (L_\alpha^2)} \,\, \mathbb E(F L_\alpha), \quad \alpha\in \mathcal I\] is the corresponding sequence of  real coefficients. Moreover, it holds 
	\[\|F\|^2_{(L)^2} = 
	\sum_{\alpha\in \mathcal I} \, f_\alpha^2 \, \,  \mathbb E L_\alpha^2 = \sum_{\alpha\in \mathcal I} \,   \frac{f_\alpha^2}{\prod_{k\in \mathbb N}{(2\alpha_k+1)}} < \infty. \]  
\end{theorem}

\begin{remark}
	We note here that the chaos  representation \eqref{FL exp}  
	of a random variable with finite second moment with respect to the underlying probability measure $\mathbb P$ can be extended also to square integrable stochastic processes, where a family of real numbers $f_\alpha$  is replaced by an appropriate family of functions with values in a certain Banach space $X$.    
	Particularly, an $X$-valued square integrable process $u=u(t,x,\omega)$ 
	can be represented as
	\begin{equation}\label{FL exp process}
	u(t,x,\omega) = \sum_{\alpha\in \mathcal I} \, u_\alpha(t,x) \,\,  L_\alpha (\boldsymbol{\xi}(\omega)).  
	\end{equation} 
	In this context,  the notation $u\in C([0,T], X)\otimes (L)^2$ means that the coefficients of the process $u$ given in the form  \eqref{FL exp process} satisfy  $u_\alpha \in C([0,T], X)$ for all $\alpha\in \mathcal I$. Additionally, the estimate    
	\[\sum\limits_{\alpha\in \mathcal I} \|u_\alpha\|^2_{C([0,T], X)} \, \mathbb E L_\alpha^2  = \sum\limits_{\alpha\in \mathcal I} \sup_{t\in [0,T]}\|u_\alpha(t)\|^2_{ X} \, \mathbb E L_\alpha^2 < \infty 
	\]holds, where the expectation  $\mathbb E L_\alpha^2$ is given by \eqref{orthogonality Lalfa}.  
	Similarly, a process $u\in C^1([0,T], X)\otimes (L)^2$ can be represented in the form \eqref{FL exp process}, where its coefficients $u_\alpha\in  C^1([0,T], X)$ for all $\alpha\in \mathcal I$. Moreover, it holds
	\begin{equation}\nonumber
	\sum\limits_{\alpha\in \mathcal I} \|u_\alpha\|^2_{C^1([0,T], X)} \, \mathbb E L_\alpha^2   
	< \infty. 
	\end{equation}
\end{remark}

\subsection{Karhunen--Lo\`eve expansion} 
The Karhunen--Lo\`eve expansion gives  a way to
represent  a stochastic process as an infinite linear combination of orthogonal functions on a bounded interval. It is used to represent spatially varying random inputs in stochastic models.   Various applications of the  Karhunen--Lo\`eve expansion can be found  in uncertainty propagation through dynamical systems with random parameter functions \cite{Constantine2}, \cite{Ghanem}, \cite{Knio}.

\begin{theorem}[Karhunen--Lo\`eve expansion theorem, \cite{Ghanem}] 
	Let $v(x, \omega)$ be a spatially varying square integrable random field defined  over the spatial domain ${\rm D}$ and  a given probability space $(\Omega, \mathcal F, \mathbb P)$,  with   
	mean  $\bar{v}(x)$ and continuous covariance function  $C_v(x_1, x_2)$. Then, $v(x, \omega)$ can be represented in the form 
	
	\begin{equation}\label{KL Exp}
	v (x, \omega)= \bar{v}(x)  + \, \sum_{k\in \mathbb N} \,\sqrt{\lambda_k} \,\,  \,    e_k(x) \, \,  Z_k(\omega)  ,
	\end{equation}where $\lambda_k$ and $e_k$, $k\in \mathbb N$ are the eigenvalues and eigenfunctions of the covariance function, i.e., they solve the integral equation  
	\begin{equation}
	\label{cov KLE}
	\int_{{\rm D}} C_v(x_1,x_2) \, \,  e_k(x_2) \, \, dx_2 = \lambda_k \,\,  e_k(x_1), \quad x_1\in {\rm D}, \, k\in \mathbb N,
	\end{equation}
	and $Z_k$ are uncorrelated zero mean random variables that have unit variance.  
\end{theorem}

For some particular covariance functions $C_v$, 
the eigenpairs $(\lambda_k, e_k)_{k\in \mathbb N}$ are known a priory, and the eigenvalues $\lambda_k$ decay as $k$ increases.  
In general, the eigenvalues and eigenvectors of the covariance function have to be calculated numerically, i.e.,   by solving  the discrete version of \eqref{cov KLE}. This  constitutes the bottleneck of the method as it requires a large  number of calculations. 

In practical applications, the series are truncated,  i.e., the random field is approximated by 
\begin{equation}\label{truncated KLE}
\tilde{v} (x, \omega)= \bar{v}(x) + \sum_{k=1}^n \,\sqrt{\lambda_k} \,\,  \,    e_k(x) \, \,  Z_k(\omega),
\end{equation}which is the finite representation with the minimal mean square error over all such finite representations. 

\begin{remark}Comparing the representation \eqref{KL Exp} with the form \eqref{FL exp} we conclude that the  random field $v$ is represented in terms of  the Wiener--Askey polynomial chaos of orders zero and one, i.e.,  it is equivalent to the representation 
	\begin{equation}
	\label{razvoj za primer} 
	v(x,\omega) 
	= \bar{v}(x) + \sum_{k\in \mathbb N} \, v_{\varepsilon^{(k)}}(x) \, L_{\varepsilon^{(k)}}({\mathbf Z}(\omega)),  
	\end{equation}since  $Z_k(\omega) =  L_{\varepsilon^{(k)}}({\mathbf Z}(\omega))$, $k\in \mathbb N$ with   ${\mathbf Z}(\omega) = (Z_1(\omega), Z_2(\omega), \dots)$ being a sequence of uncorrelated uniformly distributed zero mean random variables that have unit variance. The truncated version of the representation \eqref{razvoj za primer} is  given  by   
	\begin{equation}
	\label{razvoj za primer-truncated} 
	\tilde{v}(x,\omega)  
	= \bar{v}(x) + \sum_{k=1}^n \, v_{\varepsilon^{(k)}}(x) \, L_{\varepsilon^{(k)}}({\mathbf Z}(\omega)). 
	\end{equation}
	There,  $n$	 corresponds to the finite number of random variables of the sequence $\mathbf{Z}= (Z_1, Z_2, \dots Z_n)$ that are applied in the approximation. This is used in Section \ref{Section numerics}. 
\end{remark}

More details on methods based on stochastic polynomial representations  can be found, for example, in  \cite{Babuska}, \cite{Constantine}, \cite{Ghanem}, \cite{Knio}, \cite{XiuKarn}.

\section{Splitting methods for SPDEs }
In this section, we introduce a new numerical method which combines the   Wiener--Askey polynomial chaos expansion  \cite{XiuKarn} with deterministic splitting methods \cite{FaouOS15}. The method is then applied to problem \eqref{evolution eq}  with   non-Gaussian random inputs.  
First,   we are going to state a theorem on the  existence and uniqueness of the solution of \eqref{Eq-1.1}.  Then, we recall some convergence results of splitting methods in the deterministic setting. Finally,  we provide a convergence analysis of our approach which is the main result of this section. Thorough this section we denote $\mathcal L=A+B$.
\subsection{Existence and uniqueness of the solution}
Recall that a solution of the considered stochastic evolution problem \eqref{Eq-1.1} belongs to the space of  square integrable stochastic processes whose coefficients are    continuously differentiable deterministic functions with values in  $X$. 

\begin{definition}
	A process  $u$ is a  (classical) solution of \eqref{Eq-1.1} if  $u\in C([0,T], X)\otimes (L)^2  \, \cap \,  C^1((0,T], X)\otimes (L)^2$ and  if $u$   satisfies \eqref{Eq-1.1} pointwise. 
\end{definition}

Let the following assumptions hold: 

\begin{enumerate}
	\item[$(A1)$] Let $\mathcal L$ be  a 
	coordinatewise operator  defined on some domain $\cD(\mathcal L)$ dense in $X$, i.e.,  
	\begin{equation}\mathcal L\,  u = \sum_{\alpha\in \mathcal I} \mathcal L \, (u_\alpha) \, \, L_\alpha
	\nonumber\end{equation}for $u$ of the form \eqref{FL exp process}. Moreover, let $\mathcal L$  be  the  infinitesimal generator of a $C_0$ semigroup $(S_t)_{t\geq 0}$ of type $(M, w)$, i.e., 
	\[\|S_t\|_{L(X)} \leq M \, e^{w t}, \quad t\geq 0 \] for some $M > 0$ and $w\in \mathbb R$. 
	\item[(A2)]  Let $u^0\in X\otimes (L)^2$ and $\mathcal L u^0\in X\otimes (L)^2$, i.e.,   \[\sum_{\alpha\in \mathcal I} \, \,  \|u_\alpha^0\|^2_X \, \, 
	\mathbb E L_\alpha^2 < \infty \quad \text{and} \quad  \sum_{\alpha\in \mathcal I}  \, \|\mathcal L u_\alpha^0\|^2_X \, \, \, \mathbb E L_\alpha^2 < \infty .\]
	\item[(A3)]  The noise process is given in the form $G(t, x, \omega)= \sum\limits_{\alpha\in \mathcal I} g_\alpha(t, x) \, L_\alpha  \in C^1([0,T], X) \otimes (L)^2$, i.e.,  it holds 
	\[\sum_{\alpha\in \mathcal I} \, \,  \|g_\alpha\|^2_{C^1([0,T],X)} \, \, 
	\mathbb E L_\alpha^2 
	< \infty .\] 
\end{enumerate}

We note here that the derivative is a coordinatewise operator, i.e., for a process  $u\in C^1([0,T], X)\otimes (L)^2$  it holds
\[\frac{d}{d t} \, u (t, \omega) =  \frac{d}{d t} \,  \Big(\sum\limits_{\alpha\in \mathcal I}  u_\alpha(t) \,   L_\alpha(\boldsymbol{\xi}(\omega)) \Big) =  \sum\limits_{\alpha\in \mathcal I} \Big( \frac{d}{d t}  u_\alpha(t) \Big) \,\,  L_\alpha(\boldsymbol{\xi}(\omega)).\]
\begin{theorem}[Existence and uniqueness of the solution] 
	\label{THM Exist}
	If  the assum\-ptions $(A1)$-$(A3)$ hold, then 
	the stochastic Cauchy problem 
	\begin{equation}
	\label{Eq-opste}
	\begin{split}
	\phantom{(0)} u_t (t,  \omega) &= \mathcal L \, u(t,  \omega) + G(t,  \omega), 
	\quad 
	u(0, \omega) = u^0( \omega) 
	\end{split}
	\end{equation} has a unique solution 
	\begin{equation}
	\label{solution chaos exist}
	u (t,\omega) = \sum_{\alpha\in \mathcal I} 
	\Big(S_t u^0_\alpha + \int_0^t \, S_{t-s} \, g_\alpha(s) \, ds \Big) \, \, L_\alpha(\omega)
	\end{equation}
	in $C^1([0,T], X)\otimes (L)^2$. 
\end{theorem}
\begin{proof} We present the main steps of the proof.   We are looking for a solution in chaos representation  form 
	\[u(t, \omega) = \sum_{\alpha\in \mathcal I} \, u_\alpha(t) \, L_\alpha(\omega).\] Then, 
	by applying  the chaos expansion method,  
	the stochastic equation \eqref{Eq-opste} is transformed to  the infinite system of deterministic problems
	\begin{equation}
	\label{sistem-det}
	\begin{split}
	\frac{d}{dt} \, u_\alpha(t) &= \mathcal L \, u_\alpha(t) + g_\alpha(t), \\
	u_\alpha(0) &= u_\alpha^0
	\end{split}
	\end{equation}for all $\alpha\in \mathcal I$ that can be solved in parallel.  Since $g_\alpha\in C^1([0,T],X)$  the inhomogeneous initial value problem \eqref{sistem-det} has a solution $u_\alpha(t)\in C^1((0,T],X)$ for all $\alpha\in \mathcal I$.  Moreover, the solution $u_\alpha$ is given by 
	\begin{equation*}
	\label{solution u_alfa}
	u_\alpha (t) = S_t u^0_\alpha + \int_0^t \, S_{t-s} \, g_\alpha(s) \, ds, \quad t\in [0,T],
	\end{equation*}see \cite{Pazy}. 
	Thus, for all fixed $\alpha\in \mathcal I$ the solution  $u_\alpha(t)$ exists for all $t\in [0,T]$,  and it  is a unique classical solution on the whole interval $[0,T]$.  Also,  
	\[\frac{d}{dt} \, u_\alpha(t) = S_t \, \mathcal L u_\alpha^0 \, + \int_0^t S_{t-s} \, \, \frac{d}{ds} g_\alpha(s) \, ds + S_t \, g_\alpha(0), \quad \alpha\in \mathcal I, \, t\in [0,T]. \]
	Moreover, the series $\sum_{\alpha\in \mathcal I}u_\alpha(t) \, L_\alpha$ converges in $C^1([0,T], X) \otimes (L)^2$.  \mbox{Namely},   from the assumptions $(A1)$-$(A3)$ we obtain 
	\begin{equation*}
	\begin{split}
	\sum_{\alpha\in \mathcal I}  \|u_\alpha\|_{C^1([0,T], X)}^2 \,& \mathbb E L^2_\alpha  = 
	\sum_{\alpha\in \mathcal I} \big(\sup\limits_{t\in [0,T]} \|u_\alpha(t)\|^2_X  + \sup\limits_{t\in [0,T]} \|\frac{d}{dt}u_\alpha(t)\|^2_X \big) \,\, \mathbb E L^2_\alpha \\
	&\leq c \sum_{\alpha\in \mathcal I} \Big(\|u_\alpha^0\|^2_X  + \|\mathcal L u_\alpha^0\|^2_X + \|g_\alpha\|_{C^1([0,T],X)}^2 \Big) \, \, \mathbb E L^2_\alpha < \infty, 
	\end{split}
	\end{equation*}where 
	$c=c(M, w, T)$ is a constant depending on $M, w$ and $T$.  
\end{proof}

\begin{remark}
	If an operator $A$ is the  infinitesimal  generator of a $C_0$ semigroup and $B$ is a bounded operator then the operator $\mathcal L= A+B$ is also the   infinitesimal  generator of a $C_0$ semigroup and Theorem \ref{THM Exist} holds. In particular, Theorem \ref{THM Exist}  also holds for  analytic  semigroups. 
\end{remark}

\subsection{Splitting methods for deterministic problems}\label{Sec Splitting} 
We briefly recall the convergence of two operator resolvent splitting methods: resolvent  Lie splitting (a first-order method)  and trapezoidal resolvent splitting (a second-order method).  
Resolvent splitting methods for the time integration of abstract evolution equations were  studied in \cite{resolvent}.  The convergence properties of splitting methods for inhomogeneous evolution equations were analyzed in \cite{Ostermann12}.  
Other splitting methods were also considered in the literature. For example, exponential splitting methods for homogeneous problems with unbounded operators  were presented in  \cite{HanO09}, \cite{HanOII09}.  The inhomogeneous case was studied in \cite{FaouOS15}.   Error bounds for exponential operator splittings were further discussed in \cite{JahnL00}. 

\subsubsection{Analytic setting}      Let $X$ be an arbitrary Hilbert space with norm  \linebreak  denoted  by $\|\cdot\|$. Let $X^{*}$ be the dual space of $X$. For $t\in [0,T]$ we consider the inhomogeneous evolution equation  
\begin{equation}\label{eq:det_inh}
\begin{split}
\frac{d}{dt} u(t)  &= \mathcal L u(t) + g(t) \\& = A u(t) + B u(t) + g(t), \quad 
u(0) =  u^0,  
\end{split}
\end{equation} 
where $(\cD(\cL), \cL)$, $(\cD(A), A)$ and $(\cD(B), B)$ are linear unbounded operators in $X$  such that 
$\cD(\cL) \subseteq \cD(A) \cap \cD(B)$  and $g: [0, T ] \to X$. 
We recall the  main results from  \cite{resolvent}  and \cite{Ostermann12}. \\

Let the following assumptions hold:

\begin{enumerate}
	\item[(a1)] The operators $(\cD(\cL), \cL)$, $(\cD(A), A)$ and $(\cD(B), B)$  are maximal  \linebreak  dissipative and densely defined in $X$.
	\item[(a2)]   $\cD(\cL^{2}) \subseteq \cD(AB)$  
	\item[(a3)]  Let $0\in \rho(\cL)$, let $\cL^{-1} g(t) \in \cD(AB)$ for all $t\in [0,T]$ and 
	\[\max_{0\leq t \leq T} \|AB \cL^{-1} g(t)\| \leq c \]
	with a moderate constant $c$.  
\end{enumerate}

Recall that  an operator $(\cD(G), G)$ is maximal dissipative in $X$ if the following conditions hold:
\begin{itemize}
	\item[(i)] 
	for every $x\in \cD(G)$ there exists an element $f\in F(x)=\{h\in X^{*}: h(x) = \|x\|^{2} = \|h\|^{2}\} \subseteq X^{*}$ such that $\text{Re\;} f (Gx)\leq 0$ and 
	\item[(ii)]  $\text{range\;}(I-G)=X$.  
\end{itemize}

Since we assumed that $X$ is a Hilbert space, every maximal dissipative operator in $X$ is densely defined.  The assumption $(a1)$ is equivalent to claiming that the operators generate $C_0$ semigroups of contractions on $X$, see  \cite{Pazy}. Additionally, from $(a1)$ the following estimates hold \[\|(I-hA)^{-1}\|\leq 1 \quad \text{and} \quad  \|(I-hB)^{-1}\|\leq 1 \quad\text{for all} \, \,\,  h\ge 0.\]

We recall briefly the results from regularity theory for analytic semigroups needed in the following sections.  
\begin{theorem}{\rm (
		\cite{Lunardi})}
	Let $\mathcal L$ be the generator of an analytic semigroup and 
	let the data of  problem \eqref{eq:det_inh} satisfy 
	\begin{equation*}
	\label{assumptions regularity deterministic problem}
	u^0 \in \cD(\cL), \quad g\in C^{\theta}([0,T], X)  
	\end{equation*} for some $\theta > 0$.  Then,  the exact solution of  problem \eqref{eq:det_inh}  
	is given by the variation of constants formula 
	\begin{equation}
	\label{determ variation of constants formula}
	u(t) = e^{t\cL} u^0 + \int_0^{t} e^{(t-\tau)\cL} \, g(\tau) \, d\tau,  \quad 0 \leq t \leq T .
	\end{equation} 
	It possesses  the regularity   \[u\in C^1([0,T], X) \cap C([0,T],  D(\cL)).\] 
\end{theorem}

The same regularity is obtained if $g$ is only continuous but has a slightly improved spatial regularity, see \cite[Corollary 4.3.9]{Lunardi}. 

\begin{theorem}{\rm (\cite{Ostermann12})} \label{Thm Ostermann 12}
	Let $\mathcal L$ be the generator of an analytic semigroup. 
	Under the further assumptions 
	\begin{equation}
	\label{assumptions 2 - regularity deterministic problem}
	u^0\in D(\cL), \quad  \cL u^0 +g(0) \in D(\cL), \quad  g\in C^{1+ \theta}([0,T], X)
	\end{equation}
	for some $\theta  > 0$, the solution \eqref{determ variation of constants formula} of the evolution equation  \eqref{eq:det_inh} possesses the improved regularity
	\begin{equation}\label{regularity solution}
	u\in C^{2}([0,T], X) \cap C^{1}([0,T], D(\cL)). 
	\end{equation}
\end{theorem}

In the following we present two deterministic resolvent splitting methods \cite{Hundsdorfer}, the \emph{resolvent Lie splitting} and the  \emph{resolvent trapezoidal splitting}, that were both  applied to inhomogeneous evolution equations \eqref{eq:det_inh} in \cite{Ostermann12}.  

\subsubsection{Resolvent Lie splitting}  
The exact solution of the evolution equation \eqref{eq:det_inh} is given by the variation of constants formula \eqref{determ variation of constants formula}.   
Then, at time $t_{n+1} =t_n + h$, with a positive step size $h$, the solution can be written as 
\[u(t_{n+1}) = e^{h\cL} u(t_n)+ \int_0^{h} e^{(h-s)\cL}  \, g(t_n + s) \, d s.\]
After expanding $g(t_n+s)$ in  Taylor  form we obtain   
\[u(t_{n+1}) =  e^{h\cL} u(t_n) +   \int_0^{h} e^{(h-s)\cL} \, \Big( g(t_n) + s g'(t_n) +  \int_{t_n}^{t_n + s} (t_n + s- \tau) g''(\tau) d\tau \Big) \, ds , \]see \cite{Ostermann12}.  
For  resolvent Lie splitting, the numerical solution of \eqref{eq:det_inh}  at time $t_{n+1}$ is denoted by $u^{n+1}$ and it is given by  
\begin{equation}
{u}^{n+1} = (I-hB)^{-1} (I-hA)^{-1} (u^n + h \, g(t_n)) .
\label{Lie}
\end{equation}

\begin{theorem}[Resolvent Lie splitting,  \cite{Ostermann12}] 
	\label{splitting1} Let the assumptions $(a1)$, $(a2)$ and $(a3)$    be fulfilled and let the solution satisfy \eqref{regularity solution}. Then the resolvent Lie splitting \eqref{Lie} is first-order convergent, i.e.,  the global error satisfies the bound 
	\begin{equation}\label{global estimate deterministic resolvent Lie}
	\|u(t_n)-u^n \| \leq C h, \quad 0 \leq t_n \leq T
	\end{equation}
	with a constant $C$ that can be chosen uniformly on $[0, T ]$ and, in particular, independently of $n$ and $h$.
\end{theorem}

\begin{remark}
	The  constant $C$  in \eqref{global estimate deterministic resolvent Lie} depends on derivatives of the solution $u$ and on  $AB \cL^{-1} \, g(t)$, which are uniformly bounded on $[0,T]$ due to the asumptions of Theorem \ref{splitting1}. A detailed proof is given in \cite{Ostermann12}. 
\end{remark}

In particular, for a homogeneous evolution problem  ($g=0$) the global error  \eqref{global estimate deterministic resolvent Lie} can be estimated as 
\begin{equation*}
\label{global estimate-Lie-expanded-homogen}
\| u(t_n)-u^n \| \leq ch\, \big(\|u^0\| + \| \cL u^0\| + \|\cL^{2} u^0\|\big),  
\end{equation*}where the positive constant $c$ is independent on $n$ and $h$, see \cite{resolvent}. 

We note that the full-order convergence of Lie resolvent splitting only requires additional smoothness in space of the inhomogeneity $g$.

\subsubsection{The trapezoidal splitting} 
For a trapezoidal  splitting method,  the  \linebreak  numerical solution of \eqref{eq:det_inh} at time $t_{n+1}=t_n + h$ with a positive time step size $h$ is given by 
\begin{equation}
\label{trapezoidal method}
u^{n+1} = \Big(I - \frac h2 B\Big)^{-1} \Big(I - \frac h2 A\Big)^{-1} \Big(\Big(I + \frac h2 A\Big)\Big(I + \frac h2 B\Big) \, u^n + \frac h2 \big(g(t_n) + g(t_{n+1}) \big)\Big)
\end{equation}with $u^0=u(0)$.   

As we are considering a second-order method, we need more regularity of the solution. For analytic semigroups, this
requirement can  be  expressed in terms of the data.     
The following modification of the assumption $(a3)$ is needed:

\begin{enumerate}
	\item[(a4)]  Let $0\in \rho(\cL)$, let $\cL^{-1} g'(t) \in \cD(AB)$ for all $t\in [0,T]$ and 
	\[\max_{0\leq t \leq T} \|AB \cL^{-1} g'(t)\| \leq c \]
	with a moderate constant $c$.
\end{enumerate}

Since we assumed $X$ to be a Hilbert space, it follows from assumption (a1) that 
the estimates 
\[\|(I + hA) (I-hA)^{-1}\|\leq 1 \quad \text{and} \quad  \|(I + hB) (I-hB)^{-1}\|\leq 1 \]
hold for all $h>0$.

\begin{theorem}{\rm (\cite{Ostermann12})}  
	\label{Thm Ostermann12 2}
	Let $\mathcal L$ be the generator of an analytic semigroup. 
	If
	\begin{equation}
	\label{second order-regularity}
	\begin{split}
	g &\in C^{2+ \theta}([0,T], X), \quad \\
	u^0 &\in \cD(\cL),   \, \, \,  \cL u^0 + g(0)  \in \cD(\cL),  \, \, \,  \cL^{2} u^0 + \cL g(0) + g'(0)  \in \cD(\cL)   	\end{split}
	\end{equation}
	for some $\theta >0$, then the exact solution \eqref{determ variation of constants formula}  of the inhomogeneous evolution equation \eqref{eq:det_inh} satisfies
	\begin{equation}
	\label{regularity solution1}
	u\in C^{3}([0,T], X) \cap C^{2}([0,T], D(\cL)).
	\end{equation} 
\end{theorem}

\begin{theorem}{\rm \bf (The trapezoidal  splitting method,   
		\cite{Ostermann12})}  
	\label{splitting2} 
	Let the   \linebreak  assumptions $(a1)$, $(a2)$ and $(a4)$ be fulfilled and let the solution satisfy \eqref{regularity solution1}. Then the trapezoidal  splitting method  \eqref{trapezoidal method} is second-order convergent, i.e.,  the global error satisfies the bound
	
	\begin{equation}\label{global estimate deterministic trapezoidal}
	\| u(t_n)-u^n \| \leq C h^{2}, \quad 0 \leq t_n \leq T
	\end{equation}
	with a constant $C$ that can be chosen uniformly on $[0, T ]$ and, in particular, independently of $n$ and $ h$.
\end{theorem}

\begin{remark}
	The  constant $C$  in \eqref{global estimate deterministic trapezoidal} depends on derivatives of the solution $u$ and on  $AB \cL^{-1} \, g'(t)$, which are uniformly bounded on $[0,T]$ due to the asumptions of Theorem \ref{splitting2}. More  details are given  in \cite{Ostermann12}. 
\end{remark}

\subsection{Convergence analysis}
\label{Sec Conv Analysis}
In order to solve problem \eqref{Eq-opste} numerically, we  approximate the solution $u$  by the truncated chaos representation form 
\begin{equation}
\label{approx sol Wchaos1}
\tilde{u} = \sum_{\alpha\in \mathcal I_{m, K}} \, u_\alpha \, L_\alpha ,
\end{equation}where 
$\mathcal I_{m, K} = \{\alpha\in \mathcal I : \, \alpha = (\alpha_1, \dots, \alpha_m, 0, 0, \dots), \, |\alpha|\leq K\}$.  
Here,  $K\in  \mathbb{N}$ is the highest degree of Legendre polynomials  and $m\in  \mathbb{N}$ is the number of random variables we want to use in the approximation  \eqref{approx sol Wchaos1}.  The  $m$-dimensional random vector $\boldsymbol{\xi} = (\xi_1,\ldots,\xi_m)$ has independent and  identically distributed components   $\xi_i\sim\mathcal{U}([-1,1])$ for $i=1,\ldots,m$. 
The  choice of $m$ and $K$ influences the accuracy of the approximation. They can be chosen
so that the norm of the approximation remainder 
$u - \tilde{u}$ is smaller than a   
given tolerance. 
The sum in \eqref{approx sol Wchaos1} has   
\begin{align}\label{identity}
P = \frac{(m+K)!}{m!\ K!}
\end{align}
terms, which means that $P$ coefficients of the solution will be computed. 
Thus, only the first  $P$ equations   of the system  \eqref{sistem-det} are solved and in this way the approximation of the solution of the system is obtained.  
The global error of the proposed numerical scheme depends on the error generated by the truncation of the chaos expansion and the error of the discretisation method.  
Also, the statistics  $\mathbb{E} \tilde{u}$ and $\Var \tilde{u}$ of the approximated solution can be  calculated in terms of the obtained discretized coefficients. 
For more details on the truncation \eqref{approx sol Wchaos1}  
see for instance  \cite{XiuKarn}.   
In the following, we consider the two numerical  resolvent splitting methods,  Lie splitting and trapezoidal  splitting,  and  provide error analysis for both of them. 

\begin{theorem}[Error generated by the truncation of the Wiener--Legendre chaos expansion] \label{lemma 1 trunc} 
	Let $\tilde{u}$ denote the truncated chaos representation of the solution $u$ of the stochastic evolution problem  \eqref{Eq-opste}  given in the form  \eqref{approx sol Wchaos1}.  Let the assumptions $(A1)$-$(A3)$ 
	hold.   Then, $\tilde{u}$  approximates the solution $u$ and the approximation error satisfies the a priori bound 
	\begin{equation}
	\label{truncation error}
	\begin{split}
	\|u - \tilde{u}&\|^2_{C^1([0,T],X)\otimes (L)^2}\\
	& \leq c \sum_{\alpha\in \mathcal I \setminus \mathcal I_{m, K}} \left(\|u^0_{\alpha}\|^2_X + \|\mathcal L u^0_{\alpha}\|^2_X  + \|g_\alpha\|^2_{C^1([0,T], X)}\right) \, \mathbb E L_\alpha^2 < \infty . 
	\end{split} 
	\end{equation} 
\end{theorem}

\begin{proof}
	The approximation error due to the elimination of the higher order components of the Wiener--Legendre chaos expansion and the truncation of the noise term is obtained by 
	
	\begin{equation*}
	\label{calc lemma 1_1}
	\begin{split}
	\|u - \tilde{u}&\|^2_{C^1([0,T],X)\otimes (L)^2} 
	= \|\sum_{\alpha\in \mathcal I \setminus \mathcal I_{m, K}} u_\alpha \, \, L_\alpha\|^2_{C^1([0,T],X)\otimes (L)^2}\\
	& = \sum_{\alpha\in \mathcal I \setminus \mathcal I_{m, K}} \|u_\alpha\|^2_{C^1([0,T],X)} \, \, \mathbb E L^2_\alpha\\
	&\leq c \sum_{\alpha\in \mathcal I \setminus \mathcal I_{m, K}} \left(\|u^0_{\alpha}\|^2_X + \|\mathcal L u^0_{\alpha}\|^2_X  + \|g_\alpha\|^2_{C^1([0,T], X)}\right) \, \mathbb E L_\alpha^2, 
	\end{split}
	\end{equation*}
	which is finite by  the assumptions (A1)-(A3). 
	In the last estimate, we employed the bound derived in the proof of  Theorem \ref{THM Exist}. 
\end{proof}

\begin{theorem}[Discretization error]  \label{lemma 2 discr1} 
	Let $\tilde{u}$ denote the truncated chaos representation of the solution $u$ of the stochastic evolution problem  \eqref{Eq-opste}  given in the form  \eqref{approx sol Wchaos1}.  
	Let a square integrable  process $\tilde{u}^n_{dis}$ be given in the form 
	\begin{equation*}
	\label{s1 dis}
	\tilde{u}^n_{dis} = \sum_{\alpha\in \mathcal I_{m, K}}\,  {u}^n_{\alpha,  dis} \,\, \,  L_\alpha, 
	\end{equation*}
	where 
	its coefficients ${u}^n_{\alpha, dis}$, $\alpha\in \mathcal I_{m, K}$ are   numerical approximations of    $u_\alpha$ for  $\alpha\in \mathcal I_{m, K}$ at time $t_n= nh$ with a positive step size $h$.  Assume that the coefficients $u_\alpha$ are  sufficiently regular and  
	the approximation  
	\begin{equation}
	\label{e alpha}
	\|u_\alpha(t_n) - {u}^n_{\alpha, dis}\|_{X} \leq e_\alpha,  \qquad \alpha\in \mathcal I_{m, K}
	\end{equation}
	holds for the particular  numerical method applied.  
	Then, 
	the difference between   $\tilde{u}$ evaluated at $t_n$  and $\tilde{u}^n_{dis}$ can be estimated by the a priori bound 
	
	\begin{equation*}
	\label{Discr Bound}
	\begin{split}
	\|\tilde u(t_n) - \tilde{u}_{dis}^{n}\|_{X \otimes (L)^2}^2& \leq \sum_{\alpha\in \mathcal I_{m, K}} \, \|{u}_\alpha(t_n) - {u}^n_{\alpha, dis}\|^2_{X} \, \, \mathbb E L_\alpha^2  \\
	& \leq \sum_{\alpha\in \mathcal I_{m, K}} \, e_\alpha^2 \, \, \mathbb E L_\alpha^2  < \infty .
	\end{split}
	\end{equation*}
\end{theorem}
\begin{proof}
	From Parseval's identity and the orthogonality of the polynomial basis $\{L_\alpha\}$,  and using that the error \eqref{e alpha} 
	for a concrete numerical method, we obtain
	\begin{equation*}
	\begin{split}
	\|\tilde{u}(t_n) - \tilde{u}^n_{dis}\|_{X\otimes (L)^2}^2& = \|\sum_{\alpha\in \mathcal I_{m, K}} \, u_\alpha(t_n) L_\alpha - \sum_{\alpha\in \mathcal I_{m, K}} \, {u}^n_{\alpha, dis} L_\alpha\|_{X\otimes (L)^2}^2 \\
	&= \sum_{\alpha\in \mathcal I_{m, K}} \, \|u_\alpha (t_n) -  {u}^n_{\alpha, dis}\|_X^2  \, \, \mathbb E L^2_\alpha  \\
	&  \leq \sum_{\alpha\in \mathcal I_{m, K}} \, e_\alpha^2 \, \, \mathbb E L_\alpha^2  < \infty,\label{s2}
	\end{split}
	\end{equation*}which completes the proof. 
\end{proof}

In order to apply the splitting methods in the setting of \cite{Ostermann12},  we are going to consider the analytic case and adapt Theorem \ref{lemma 2 discr1}.  We replace the assumption  $(A1)$ with the assumption:
\begin{enumerate}
	\item[(B1)] Let  $(A, \mathcal D(A))$, $(B, \mathcal D (B))$ and $(\mathcal L, \mathcal D (\mathcal L))$ be coordinatewise operators that generate analytic semigroups of contractions on $X$.    
	Let $\cD(\cL^2) \subseteq \cD(AB)$. 
\end{enumerate}

Further, for the case of the resolvent Lie splitting we replace the assumptions $(A2)$ and $(A3)$  by:   
\begin{enumerate}
	\item[(B2)]  The noise process given by 
	\begin{equation}
	\label{G assummed}
	G = \sum\limits_{\alpha\in \mathcal I} g_\alpha \, L_\alpha
	\end{equation} belongs to  $C^{1+\theta} ([0,T], X) \otimes (L)^2$ for some $\theta>0$, i.e., 
	\begin{equation}   
	\sum_{\alpha\in \mathcal I} \, \,  \|g_\alpha\|^2_{C^{1+\theta }([0,T],X)} \, \, 
	\mathbb E L_\alpha^2 < \infty 
	\label{uslovi g lie}
	\end{equation} holds. 
	\item[(B3)]   Let $u^0 \in \mathcal D(\mathcal L) \otimes (L)^2$ and $\mathcal L u^0 + G(0)  \in \mathcal D(\mathcal L) \otimes (L)^2$,  i.e.,   
	\begin{equation*}
	\label{pp a2 lie}
	\sum_{\alpha\in \mathcal I} \, \,  \|u_\alpha^0\|^2_{D(\mathcal L) } \, \, 
	\mathbb E L_\alpha^2 < \infty \quad \text{and} \quad  \sum_{\alpha\in \mathcal I}  \, \|\mathcal L u_\alpha^0 + g_\alpha(0)\|^2_{{D(\mathcal L) }} \, \, \, \mathbb E L_\alpha^2 < \infty .
	\end{equation*}
	\item[(B4)] Let $0\in \rho(\cL)$, let $\cL^{-1} G(t) \in \cD(AB) \otimes (L)^2$ for all $t\in [0,T]$ and let  the coefficients $g_\alpha$ of $G$ given by \eqref{G assummed}, satisfy the estimate 
	\[\max_{0\leq t \leq T} \|AB \cL^{-1} g_\alpha(t)\| \leq \, c_\alpha, \quad 0\leq t \leq T\]
	with a moderate constant $c_\alpha$ for each $\alpha\in \mathcal I$.  
\end{enumerate}

Note that,  under these assumptions, the existence theorem, Theorem \ref{THM Exist}, still holds. Particularly, for the resolvent  Lie splitting it reads: 
\begin{theorem}
	Let $\mathcal L$ be the generator of an analytic semigroup. Under the 
	assumptions $(B2)$ and $(B3)$, the solution \eqref{solution chaos exist} of 
	the stochastic evolution problem 
	\eqref{Eq-1.1} posseses the  improved regularity 
	\begin{equation}
	\label{regularity Lie chaos}
	u\in  C^2([0,T], X)\otimes (L)^2 \cap 
	C^1([0,T], \mathcal D(\mathcal L)) \otimes (L)^2 .
	\end{equation}
\end{theorem}
\begin{proof}
	By the method of chaos expansion, the stochastic evolution problem 
	\eqref{Eq-1.1} transforms to the system of deterministic problems \eqref{sistem-det}. From $(B2)$ and $(B3)$ it follows that $u_\alpha^0$ and $g_\alpha$ for each  $\alpha\in \mathcal I$ satisfy the assumptions \eqref{assumptions 2 - regularity deterministic problem}. After applying Theorem \ref{Thm Ostermann 12}  we obtain the improved  regularity  $u_\alpha\in C^{2}([0,T], X) \cap C^{1}([0,T], D(\cL))$, $\alpha\in \mathcal I$. 
\end{proof}
\begin{theorem}[Discretization error, the resolvent Lie splitting] 
	\label{lemma 2a discr lie}
	Let  the assumptions $(B1)$-$(B4)$ be fulfilled.  
	Then, for  the resolvent Lie splitting, Theorem \ref{lemma 2 discr1}  holds with 
	\begin{equation*}
	e_\alpha \leq c_\alpha\, h, \quad \alpha\in \mathcal I_{m, K}.
	\end{equation*}
	The constants  $c_\alpha$     can be chosen uniformly on $[0, T]$ and, in particular, independently of $n$ and $ h$. 
\end{theorem}
\begin{proof}
	The coefficients $u_\alpha$, for each  $\alpha\in \mathcal I_{m, K}$  are the exact solutions of the deterministic initial value problems \eqref{sistem-det} and $u^n_{\alpha, dis}$ are their numerical approximations obtained by the resolvent Lie splitting \eqref{Lie}.  Moreover, $u_\alpha$ satisfy the assumptions \eqref{assumptions 2 - regularity deterministic problem} for all $\alpha\in \mathcal I$.  
	Thus, we can apply Theorem \ref{splitting1} to each initial value problem  \eqref{sistem-det} and obtain the global  estimate \eqref{global estimate deterministic resolvent Lie}  for each $\alpha\in \mathcal I_{m, K}$, i.e. $e_\alpha \leq c_\alpha  h$, for $\alpha\in \mathcal I_{m, K}$. This leads to the desired result.  
\end{proof}

In the case of the trapezoidal resolvent splitting, we need the following additional assumptions:  
\begin{enumerate}
	\item[(B5)] 
	The noise process $G$ given by \eqref{G assummed} belongs to $C^{2+\theta} ([0,T], X) \otimes (L)^2$ for some $\theta >0$. 
	\item[(B6)] 
	Let $ \mathcal L^2 u^0 + \mathcal L G(0) + G'(0)  \in \mathcal D(\mathcal L) \otimes (L)^2$, i.e.,   
	\[  \sum_{\alpha\in \mathcal I}  \|\mathcal L^2 u_\alpha^0 + \mathcal L g_\alpha(0) + g'_\alpha(0)\|^2_{{D(\mathcal L) }} \,\,  \mathbb E L_\alpha^2 < \infty .\]
	\item[(B7)] Let $0\in \rho(\cL)$, let $\cL^{-1} G'(t) \in \cD(AB) \otimes (L)^2$ for all $t\in [0,T]$ and  let  the coefficients $g_\alpha$ of $G$ given by \eqref{G assummed}, satisfy the estimate  
	\[\max_{0\leq t \leq T} \|AB \cL^{-1} g_\alpha'(t)\| \leq \, c_\alpha\]
	with a moderate constant $c_\alpha$ for each $\alpha\in \mathcal I$. 
\end{enumerate}

\begin{theorem}
	Let $\mathcal L$ be the generator of an analytic semigroup. Under the 
	assumptions $(B3)$, $(B5)$ and $(B6)$, the solution \eqref{solution chaos exist} of 
	the stochastic evolution problem 
	\eqref{Eq-1.1} posseses the  improved regularity 
	\begin{equation*}
	\label{regularity trapezoidal chaos}
	u\in  C^3([0,T], X)\otimes (L)^2 \cap 
	C^2([0,T], \mathcal D(\mathcal L)) \otimes (L)^2 .
	\end{equation*}
\end{theorem}
\begin{proof}
	The method of chaos expansion transforms the stochastic evolution problem 
	\eqref{Eq-1.1}  to the system of deterministic problems \eqref{sistem-det}. From $(B3)$, $(B5)$ and $(B6)$ it follows that $u_\alpha^0$ and $g_\alpha$ for each  $\alpha\in \mathcal I$ satisfy the assumptions \eqref{second order-regularity}. Then,   the improved  regularity  $u_\alpha\in C^{3}([0,T], X) \cap C^{2}([0,T], D(\cL))$, for $\alpha\in \mathcal I$ follows from  Theorem \ref{Thm Ostermann12 2}.  
\end{proof}

\begin{theorem}[Discretization error, the trapezoidal resolvent  splitting] 
	\label{lemma 2b discr trapezoidal}
	Let  the assumptions $(B1)$, $(B3)$ and  $(B5)$-$(B7)$ be fulfilled.   
	Then, for  the trapezoidal resolvent  splitting, Theorem \ref{lemma 2 discr1}  holds with
	\begin{equation*}
	e_\alpha \leq c_\alpha\, h^{2}, \quad \alpha\in \mathcal I_{m, K}.
	\end{equation*}
	The constants  $c_\alpha$    can be chosen uniformly on $[0, T]$ and, in particular, independently of $n$ and $ h$.   

\end{theorem}
\begin{proof}
	From the assumptions it follows that the coefficients $u^0_\alpha$ and $g_\alpha$ satisfy  \eqref{second order-regularity} for each $\alpha\in \mathcal I_{m, K}$.  
	We apply the trapezoidal resolvent splitting  \eqref{trapezoidal method}  
	in order to obtain the approximation $u^n_{\alpha, dis}$ of the exact solution $u_\alpha(t_n)$ evaluated at $t_n$   of the  initial value problem \eqref{sistem-det} for each $\alpha\in \mathcal I_{m, K}$. Thus, by  Theorem \ref{splitting2} we obtain the global error estimate  \eqref{global estimate deterministic trapezoidal}, i.e. $e_\alpha \leq c\,  h^{2}$ for each $\alpha\in \mathcal I_{m, K}$.    
\end{proof}

Denote by $\frac12 \Delta$ the constant on the right hand side of the estimate \eqref{truncation error} obtained in   Theorem \ref{lemma 1 trunc}.    The full error estimates of the  Wiener--Legendre chaos expansion combined with the two splitting methods are given in the following theorem. 
\begin{theorem}[Full error estimate]\label{THM Split} $ $ 
	\begin{enumerate}
		\item[$(1)$]  
		Let the assumptions of  Theorem \ref{lemma 2a discr lie} hold.  
		Then, the full error  \linebreak 
		estimate of the  Wiener--Legendre chaos expansion  combined with the  resolvent Lie splitting  satisfies the following bound
		\begin{equation}
		\label{global error-chaos lie}
		\|u (t_n) - \tilde{u}^n_{dis}\|^2_{X\otimes (L)^2}  \leq \Delta + c\, h^{2}. 
		\end{equation} 
		\item[$(2)$] 
		Let the assumptions of   Theorem  \ref{lemma 2b discr trapezoidal}  hold.  Then, the full error  \linebreak  estimate of the  Wiener--Legendre chaos expansion  combined with 
		the trapezoidal resolvent  splitting satisfies the bound
		\begin{equation}
		\label{global error-chaos trapezoidal}
		\|u (t_n)- \tilde{u}^n_{dis}\|^2_{X\otimes (L)^2}   \leq \Delta + c \, h^{4}. 
		\end{equation}  
	\end{enumerate} 
\end{theorem}
\begin{proof}  
	The full error estimate reads 
	\begin{equation*}
	\begin{split}
	\|u(t_n) &- \tilde{u}^n_{dis}\|^2_{X\otimes (L)^2} 
	= \|\sum_{\alpha\in \mathcal I} u_\alpha(t_n) L_\alpha - \sum\limits_{\alpha\in \mathcal I_{m, K}} \, u^n_{\alpha, dis} \, L_\alpha \|^2_{X\otimes (L)^2}\\
	&= \|\sum\limits_{\alpha\in \mathcal I \setminus \mathcal I_{m, K}} \, u_{\alpha}(t_n) \, L_\alpha \,  + \,  \sum_{\alpha\in \mathcal I_{n, K}} (u_\alpha(t_n) - u^n_{\alpha, dis}) \, \,  L_\alpha\|^2_{X\otimes (L)^2}  \\
	&\leq 2 \sum_{\alpha\in \mathcal I \setminus \mathcal I_{m, K}}  \|u_{\alpha}(t_n)\|^2_{X} \,  \mathbb E L_\alpha^2   + 2 \sum_{\alpha\in \mathcal I_{m, K}} \|u_\alpha(t_n) - u^n_{\alpha, dis}\|^2_{X} \, \,  \mathbb E L_\alpha^2   \\
	&\leq   \Delta + 2  \sum_{\alpha\in \mathcal I_{m, K}} \, e_\alpha^2 \,  \mathbb E L_\alpha^2 
	\end{split}
	\end{equation*}by 
	the triangle inequality and the orthogonality property \eqref{orthogonality Lalfa}.  We apply 
	Theorem \ref{lemma 1 trunc} to the first term. In the case of  the resolvent Lie splitting, the estimate \eqref{global error-chaos lie} follows after  applying  Theorem	\ref{lemma 2a discr lie},  while in case of   the trapezoidal resolvent splitting,  Theorem \ref{lemma 2b discr trapezoidal} leads to the desired estimate \eqref{global error-chaos trapezoidal}.  
\end{proof}

\section{Numerical Results}
\label{Section numerics}
In this section, we  validate the proposed method and the convergence analysis presented in the previous section.  For this purpose, we  consider the two-dimensional problem 
\begin{equation}\label{heateq_v}
u_t = \mathcal {L}u + v+1, \quad u(0) = 0 , \quad  
u\big|_{\partial {\rm D}} = 0, 
\end{equation}
where the operator $\mathcal {L}$ is defined by $\mathcal {L}u = (A+B) u = (a u_x)_x  + (b u_y)_y$ over the spatial domain ${\rm D}=[-1,1]^2$ with state variables $x$ and $y$, spatial non-Gaussian noise $v$ given in the form \eqref{KL Exp} and $t\in [0,T]$ for some $T>0$. This problem is an example of the problem  class \eqref{Eq-1.1} with zero initial and boundary conditions. 
The solution $u$ of the considered problem \eqref{heateq_v} is given in  its polynomial chaos  representation \eqref{FL exp process} and  approximated by a truncated  expansion  \eqref{approx sol Wchaos1}  in terms of Fourier--Legendre polynomials. The truncation procedure is explained in detail in Section \ref{Sec Conv Analysis}.  

Consider the set of multiindices $\mathcal I_{m, K}\subset \mathcal I$, i.e.,  \[\mathcal I_{m, K} = \{\alpha\in \mathcal I : \, \alpha = (\alpha_1, \dots, \alpha_m, 0, 0, \dots), \, |\alpha|\leq K\}.\]
In this section, elements $\alpha\in \mathcal I_{m, K}$ will be denoted as $m$-tuples  $\alpha=(\alpha_1, \dots, \alpha_m)$, 
omitting the components $\alpha_j=0$, $j\geq  m+1$.   
Moreover, we set
$$
\varepsilon^{(k)} = (\varepsilon^{(k)}_1,\ldots,\varepsilon^{(k)}_m), \qquad  \varepsilon^{(k)}_j = \delta_{kj}.
$$
For fixed $m\in \mathbb{N}$ we consider an index function \[K_m: \, \, \mathcal{I}_{m,K} \rightarrow \{0, 1, \dots,  P-1\}\] which enumerates multi-indices $\alpha=(\alpha_1, \alpha_2, \dots, \alpha_m)\in \mathcal I_{m,K}$. The function $K_m$ is a bijection and each $\alpha\in \mathcal I_{m,K}$ corresponds to a unique  $K_m(\alpha)=p \in \{0,1, \dots P-1\}$.  
For our purpose, we define the function $K_m$ by  
\[\begin{split}
&K_m(0,0,\ldots,0,0) = 0,\\
&K_m(\varepsilon^{(k)}) =  k\qquad\text{for}\ 1\leq k \leq m,\\
&K_m(\varepsilon^{(k)}+\varepsilon^{(\ell)}) =  m+(m-1)+\ldots+(m-k+1) +\ell \qquad\text{for}\ 1\leq k \leq\ell \leq m,\\
&\dots\\
&K_m(0,0,\dots,0,K)=P-1. 
\end{split}\]

We use the index function $K_m$ to enumerate the Fourier--Legendre polynomials $L_\alpha$ for each $\alpha\in \mathcal I_{m,K}$.  
Thus, we  denote by $(\Phi_p)_{p\in\{0,1, \dots ,P-1\}}$   the  ordered Fourier--Legendre polynomials 
\begin{equation*}
\label{reordered F-Legendre polynomials}
\Phi_p (\boldsymbol{\xi}(\omega))  =  \Phi_{K_m(\alpha)} (\boldsymbol{\xi}(\omega))  = L_\alpha  (\boldsymbol{\xi}(\omega))     
\end{equation*}for $p=K_m (\alpha)$, $\alpha\in \mathcal I_{m,K}$, where we use the definition  \eqref{Fourier-Legendre pol} of the Fourier--Legendre polynomials. 
For example, following the just introduced notation, we have $\Phi_0 (\boldsymbol{\xi}(\omega))= L_{(0,0, \dots,0)}(\boldsymbol{\xi}(\omega)) = 1$ and \[\Phi_k (\boldsymbol{\xi}(\omega)) = L_{\varepsilon^{(k)}} (\boldsymbol{\xi}(\omega)) = \xi_k(\omega)  \quad \text{for} \, \, \, 1\leq k \leq m .\] 
Also, by applying   the definition of the Legendre polynomials  \eqref{Legendre polynomials def} we have  \[\Phi_{m+1} (\boldsymbol{\xi}(\omega)) = L_{(2, 0, \dots, 0)}(\boldsymbol{\xi}(\omega)) = p_2(\xi_1 (\omega)) = \frac32 \xi_1^2(\omega) - \frac12 ,\]  as well as 
\[ \Phi_{m+2} (\boldsymbol{\xi}(\omega)) = L_{(1, 1,0, \dots, 0)}(\boldsymbol{\xi}(\omega)) = p_1(\xi_1 (\omega)) p_1(\xi_2 (\omega)) = \xi_1(\omega) \xi_2(\omega) .\]Moreover, it holds  
\[\Phi_{P-1} (\boldsymbol{\xi}(\omega)) = L_{(0,0, \dots, 0,K)}(\boldsymbol{\xi}(\omega)) = p_K(\xi_n (\omega)).\]

In the next step, we represent the solution $u$ of problem \eqref{heateq_v}
by its truncated polynomial chaos expansion  \eqref{approx sol Wchaos1} and the noise term by its  representation \eqref{razvoj za primer-truncated}. Inserting the representations in \eqref{heateq_v} gives  
\[\sum_{\alpha\in \mathcal I_{m, K}} (u_\alpha)_t \, \,L_\alpha = \sum_{\alpha\in \mathcal I_{m, K}} \mathcal L u_\alpha  \, \,L_\alpha  + \bar{v} + 1 + \sum_{j=1}^m \sqrt{\lambda_j} \,\, e_j \,\, Z_j . \] 
By performing a Galerkin projection we obtain 
\[\begin{split} 
\sum_{\alpha\in \mathcal I_{m, K}} &(u_\alpha)_t \, \mathbb E(L_\alpha L_\beta) =\\ =&\sum_{\alpha\in \mathcal I_{m, K}} \mathcal L u_\alpha  \, \,\mathbb E(L_\alpha L_\beta) + (\bar{v} + 1) \mathbb E L_\beta + \sum_{j=1}^m \sqrt{\lambda_j} \, e_j \,\,\mathbb E(Z_j L_\beta) 
\end{split}\]
for $\beta\in \mathcal I_{m,K}$.  

Then, by applying the properties of the Fourier--Legendre polynomials \eqref{expectation Lalpha} and \eqref{orthogonality Lalfa}, 
we obtain a  system of deterministic equations \eqref{sistem-det}. \linebreak Particularly,

\begin{enumerate}
	\item[(i)] for $|\alpha|=0$: \qquad 
	\begin{equation}\label{duzina alpha 0}
	(u_{(0,0, \dots,0)})_t =\mathcal L u_{(0,0, \dots,0)} + \bar{v} + 1, \quad u_{(0,0, \dots,0)} (0) = 0, \quad u_{(0,0, \dots,0)}\big|_{\partial {\rm D}}  = 0
	\end{equation}
	\item[(ii)] for $|\alpha|=1$, i.e., $\alpha=\varepsilon^{(k)}$, $1\leq k\leq m$: \qquad 
	\begin{equation}\label{duzina alpha=1}
	(u_{\varepsilon^{(k)}})_t =\mathcal L u_{\varepsilon^{(k)}} +  \sqrt{\lambda_k} \, e_k\,,   \quad u_{\varepsilon^{(k)}} (0) = 0, \quad u_{\varepsilon^{(k)}}\big|_{\partial {\rm D}}  = 0
	\end{equation}
	\item[(iii)] for $|\alpha|>1$:
	\begin{equation}\label{duzina alpha>1}
	(u_\alpha)_t = \mathcal L u_\alpha, \quad u_{\alpha} (0) = 0, \quad u_{\alpha}\big|_{\partial D}  = 0. 
	\end{equation} 
\end{enumerate}

From \eqref{duzina alpha>1} we clearly deduce that $u_\alpha\equiv 0$ for $|\alpha|>1$. In the calculations we also used $\mathbb{E}( L_{(0,0,\ldots,0)} Z_j) = \mathbb{E} Z_j = 0$ for $j\geq 1$ and 
\[\mathbb E (Z_j \, L_\beta) = \mathbb E (p_1(Z_j) \, L_\beta) = \mathbb E (L_{\varepsilon^{(j)} }  \, L_\beta) = \delta_{\beta, \varepsilon^{(j)}} \, \mathbb E  L_{\varepsilon^{(j)}}^2 
= \delta_{\beta, \varepsilon^{(j)}} \cdot \frac13.\]This particularly implies
\[\sum_{j=1}^m \sqrt{\lambda_j} \, e_j\, \mathbb{E}(Z_j L_{\varepsilon^{(k)}} ) = \sqrt{\lambda_k} \, e_k \quad \text{for} \quad 1\leq k \leq m, \] which was used in equation   \eqref{duzina alpha=1}. 

The obtained system \eqref{duzina alpha 0}, \eqref{duzina alpha=1} and \eqref{duzina alpha>1} can be represented in terms of the index  function $K_m$, i.e., in the form  
\begin{equation}
\label{projected_eq}
(u_p)_t = 
\mathcal{L} u_p  + g_p, \quad u_p(0) = 0,  \quad 
u_p\big|_{\partial D}  = 0 
\end{equation}for $0\leq p \leq P-1$, where each $p$ corresponds to an $\alpha \in \mathcal{I}_{m,K}$
Each equation in \eqref{projected_eq} has the form of an  inhomogeneous deterministic initial value problem, where  the inhomogeneities $g_p$ are given by:  $g_0= \bar{v} +1$  and  $g_p = \sqrt{\lambda_p} \, e_p$  for $1\leq p \leq m$ and $g_p=0$ for  $m< p \leq P-1$.  

One way to approximate numerically a problem of the form
\begin{equation*}
\label{heateq_general}
u_t = ({A+B})u + g, \quad u(0) = u^0,  \quad 
u\big|_{\partial \mathcal{D}} = 0
\end{equation*}
with ${\rm D}=[-1,1]^2$ is  to define a grid consisting of $N\times N$ equidistant  \linebreak  computational points and define the discrete operators  ${A}_s$ and ${B}_s$ by
\begin{equation*}
\begin{split}
&({A}_s u^{\mathrm{dis}})_{i,j} = \frac{1}{2s}\Big(\frac{\mathrm{d}}{\mathrm{d}x} a_{i,j} \, (u^{\mathrm{dis}}_{i+1,j}-u^{\mathrm{dis}}_{i-1,j}) \Big) + \frac1{s^{2}}\Big( a_{i,j} \, (u^{\mathrm{dis}}_{i+1,j}-2u^{\mathrm{dis}}_{i,j}+u^{\mathrm{dis}}_{i-1,j}) \Big),\\
&({B}_s u^{\mathrm{dis}})_{i,j} = \frac{1}{2s}\Big(\frac{\mathrm{d}}{\mathrm{d}y} b_{i,j} \, (u^{\mathrm{dis}}_{i,j+1}-u^{\mathrm{dis}}_{i,j-1}) \Big) + \frac1{s^{2}}\Big( b_{i,j} \, (u^{\mathrm{dis}}_{i,j+1}-2u^{\mathrm{dis}}_{i,j}+u^{\mathrm{dis}}_{i,j-1}) \Big), 
\end{split}
\end{equation*}
where 
\[ \begin{split}
\frac{\mathrm{d}}{\mathrm{d}x}a_{i,j} &= \frac{\mathrm{d}}{\mathrm{d}x}a(is,js),  \ \text{ and } a_{i,j} = a(is,js),\\
\frac{\mathrm{d}}{\mathrm{d}y}b_{i,j} &= \frac{\mathrm{d}}{\mathrm{d}y}b(is,js), \ \text{ and } b_{i,j} = b(is,js)
\end{split}\]
for $i,j=1,\ldots,N$ and $s=2/(N+1)$. Due to the homogeneous Dirichlet boundary conditions we have:
\begin{align*}
u^{\mathrm{dis}}_{0,j} = u^{\mathrm{dis}}_{N+1,j} = u^{\mathrm{dis}}_{i,0} = u^{\mathrm{dis}}_{i,N+1} = 0
\end{align*}
for all $i,j = 0,\ldots,N+1$.
By setting $\mathcal{L}_s = A_s + B_s$ we obtain the discretized problem
\begin{equation*}
\frac{\mathrm{d}}{\mathrm{d}t} u^{\mathrm{dis}} = \mathcal{L}_s u^{\mathrm{dis}} + g_s(t), \quad u^{\mathrm{dis}}(0) = 0, 
\end{equation*}
where $g_s$ denotes the discretization of the inhomogeneity $g$. 

Note that the number $P$ of partial differential equations one has to solve  in \eqref{projected_eq} increases fast due to the factorials occurring in \eqref{identity}.
Since  $g_p =  0 $ for all  $m<p\leq P-1$,  $u_p = 0$ is consequently the solution of the $p$th partial differential equation of \eqref{projected_eq}. Therefore, we only have to solve the first  $m+1$ partial differential equations instead of all $P$. Further, we see that the solution does not depend  on the highest degree $K$ of the $m$-dimensional Legendre polynomials. 

Let $u_p^{n}$ denote the numerical solution $u_p$ at time  $t_n = hn$ and $g_p^n$ the function $g_p$ evaluated at  time $t_n$.
By setting
\begin{equation}\label{LIE_SPL}
u_p^{n+1} = (I-h{A_s})^{-1}(I-h{B_s})^{-1}\big(u_p^{n}+h g_p^n\big)
\end{equation}
the Lie resolvent splitting method is defined, see \eqref{Lie}.   


The trapezoidal splitting method is given by
\begin{equation*}\label{TRAP_SPL}
u_p^{n+1} = \Big(I-\frac{h}{2} {B_s}\Big)^{-1}\Big(I-\frac{h}{2} {A_s}\Big)^{-1} \Big[  \Big(I+\frac{h}{2} {A_s}\Big)\Big(I+\frac{h}{2} {B_s}\Big)  u_p^{n}+\frac{h}{2} \big(g_p^n + g_p^{n+1}\big)\Big],
\end{equation*}
see \eqref{trapezoidal method}.

In our numerical experiment, we consider \eqref{heateq_v} with constant coefficients $a(x,y) = b(x,y) = 1$  for all $(x,y) \in {\rm D}=[-1,1]^2$ and set $T=1$.
Note that for some $p\in \{0,\ldots,m\}$ the inhomogeneities $g_p$ might be incompatible with the boundary conditions at the corners of the spatial domain ${\rm D}$. Such an incompatibility results in order reduction, see \cite{Hell15}. This in particular leads to large errors near the corners of ${\rm D}$. To overcome this problem, we apply the \emph{modified Lie resolvent splitting}~\cite{Hell15} in this situation.

For $p\in \{0,\ldots,m\}$, let $u_p$ be the solution of the partial differential equation \eqref{projected_eq}.
Let $I=\{1,2,3,4\}$ be the set of indices of the corners of the spatial domain ${\rm D}$. They are enumerated from 1 to 4 counter-clockwise starting from the corner with coordinates $(-1,-1)$.
Suppose that the inhomogeneity $g_p$ does not vanish at the corners $I_p\subset I$. Let $g_{p,i}(t)$ denote the value of the function $g_p$ at corner $i\in I_p$ and time $t\geq 0$. For $g_{p,i}(0) \neq 0$ we set \[f_i = \frac{P_i\, g_p(0)}{g_{p,i}(0)} ,\] where the polynomials $P_i$ are given by
\begin{align*}
&P_1 = \frac{1}{4}(x-1)(y-1),  &P_2 = -\frac{1}{4}(x+1)(y-1),\\
&P_3 = \frac{1}{4}(x+1)(y+1),  &P_4 = -\frac{1}{4}(x-1)(y+1).
\end{align*}
These four polynomials form a partition of unity.

Let $v_i$ be the solution of the stationary problem
\begin{equation}
\mathcal{L}v_i = f_i \text{ in } {\rm D},\qquad
v_i\big|_{\partial {\rm D}} = 0,  \nonumber
\end{equation}
for $i \in I_p$. Note that $v_i$ can be computed once and for all.
Then, let
\begin{equation*}
\tilde{g}_p(t) = g_p(t) + \sum_{i \in I_p} g_{p,i}'(t) \,\, v_i - g_{p,i}(t) \,\, f_i, \quad \tilde{u}_{p,0} = u_{p}(0) + \sum_{i \in I_p} g_{p,i}(0) \,\, v_i
\end{equation*}
and apply the resolvent Lie splitting to the problem
\begin{equation*}
(\tilde{u}_p)_t =  \mathcal{L}\tilde{u}_p(t) + \tilde{g}_p(t),  \qquad
\tilde{u}_p(0) = \tilde{u}_{p,0},\quad  \tilde{u}_p\big|_{\partial {\rm D}} = 0.
\end{equation*}
By setting
\begin{equation}\label{MOD_LIE_SPL}
u_{p}^{n,\mathrm{mod}} = \tilde{u}_p^n-\sum_{i\in I_p} g_{p,i}(nh) \, v_i \quad \text{for} \quad n\in \mathbb{N},
\end{equation}
we obtain the modified splitting scheme. Note that in our case $g_{p,i}'(t) = 0$ for all $i \in I_p$ and for all $p=0,\ldots,m$ since none of the inhomogeneities $g_p$ is time dependent.

In the implementation, the set $I_p$  for $p=0,\ldots,m$ is constructed by checking the values of the inhomogeneities $g_p$ at the corners, i.e.,
\begin{align*}
I_p = \big\{ i \in \{1,2,3,4\} \big| \  \left|g_{p,i}(0)\right| \geq \texttt{TOL}\big\}
\end{align*}
for a user chosen tolerance $\texttt{TOL}$. If $I_p = \emptyset$, the standard Lie resolvent splitting given in (\ref{LIE_SPL}) is applied.

In the following, we consider problem \eqref{heateq_v} with $v$ given by \eqref{KL Exp} with covariance function
\[C_v(\mathbf{x},\mathbf{y}) =  \mathrm{exp}\{ -\| \mathbf{x}-\mathbf{y}\|^2\}.\]

The reference solution $u^{\text{ref}}_p$  at time $t$ is calculated according to
\begin{equation*}
u^{\text{ref}}_p (t) = \mathrm{exp}( t\, \mathcal{L}) u_p(0) + t \varphi_1 (t\, \mathcal{L}) \, g_p ,
\end{equation*}
where $\varphi_1(z) = \frac{\mathrm{exp}(z)-1}{z}$ and $\mathrm{exp}(\cdot)$ denotes the matrix exponential.  
In all the  examples shown we fix the highest degree of ordered Fourier--Legendre polynomials to $K=3$ and use a maximal number of $m=120$ uncorrelated zero-mean random variables $Z_j$ used in the truncated  Karhunen--Lo\`eve \linebreak expansion  \eqref{truncated KLE}. If not stated explicitly, we fix the number of computational points to $N\times N = 40 \times 40$. \\


\begin{figure}[h!]
	\begin{minipage}{\linewidth}
		\centering
		\begin{tabular}{cc}
			
			\begin{tikzpicture}
			\pgfplotsset{every tick label/.append style={font=\small}}
			\begin{axis}[
			width=0.5\textwidth,
			height=0.5\textwidth,
			axis on top,
			xmin=-1,
			xmax=1,
			ymin=-1,
			ymax=1,
			colorbar horizontal,
			colormap/jet,
			point meta min=0,
			point meta max=0.036279568514356]
			
			\addplot graphics [xmin=-1,xmax=1,ymin=-1,ymax=1] {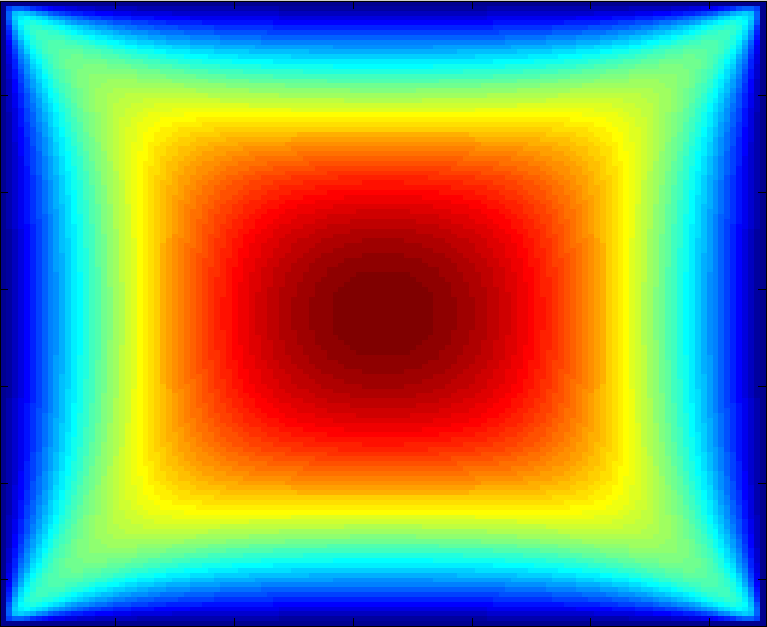};
			\end{axis}
			\end{tikzpicture}&
			
			\begin{tikzpicture}
			\pgfplotsset{every tick label/.append style={font=\small}}
			\begin{axis}[
			width=0.5\textwidth,
			height=0.5\textwidth,
			axis on top,
			xmin=-1,
			xmax=1,
			ymin=-1,
			ymax=1,
			colorbar horizontal,
			colormap/jet, 
			point meta min=0,
			point meta max=0.00145493270165176]
			
			\addplot graphics [xmin=-1,xmax=1,ymin=-1,ymax=1] {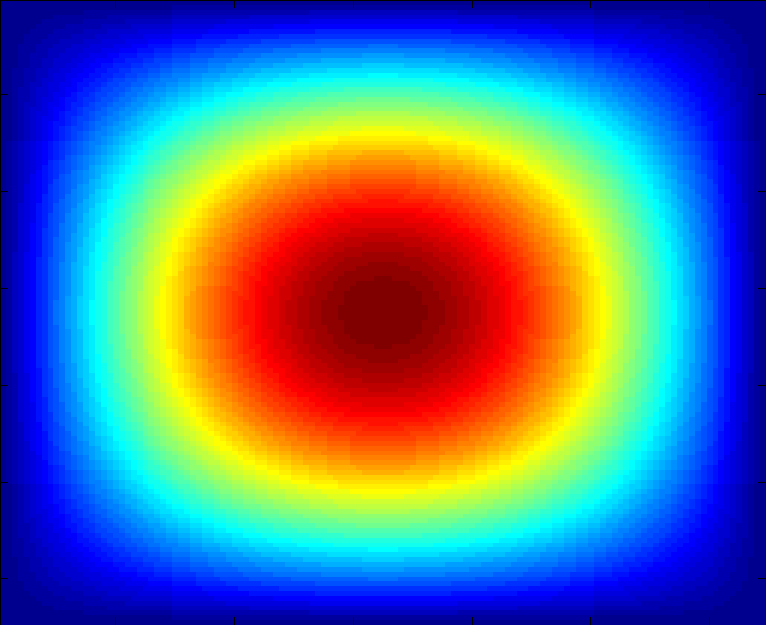};
			\end{axis}
			\end{tikzpicture}
		\end{tabular}
		\caption{\small Pointwise error of $u_0$ over the domain $\mathcal{D}=[-1,1]^2$ for the Lie splitting (left) and the modified Lie splitting (right).}
		\label{lie_vs_mod_lie}
	\end{minipage}
\end{figure}

Figure \ref{lie_vs_mod_lie} illustrates the impact of the modification of the Lie resolvent splitting method. The figure shows the pointwise error of the numerical solution at time  $T=1$, i.e.,  $|u_{0}(T) - u^{\text{ref}}_0(T)|$ over the spatial domain ${\rm D}=[-1,1]^2$ when calculated with the Lie splitting and the modified Lie splitting given in \eqref{LIE_SPL} and \eqref{MOD_LIE_SPL}, respectively. The pointwise error of the solution $u_0$ is not only reduced at all the four corners of the domain ${\rm D}$ but also approximately decreases by an order of magnitude. \\

%
%
\begin{figure}
	\centering
\includegraphics[width=\linewidth]{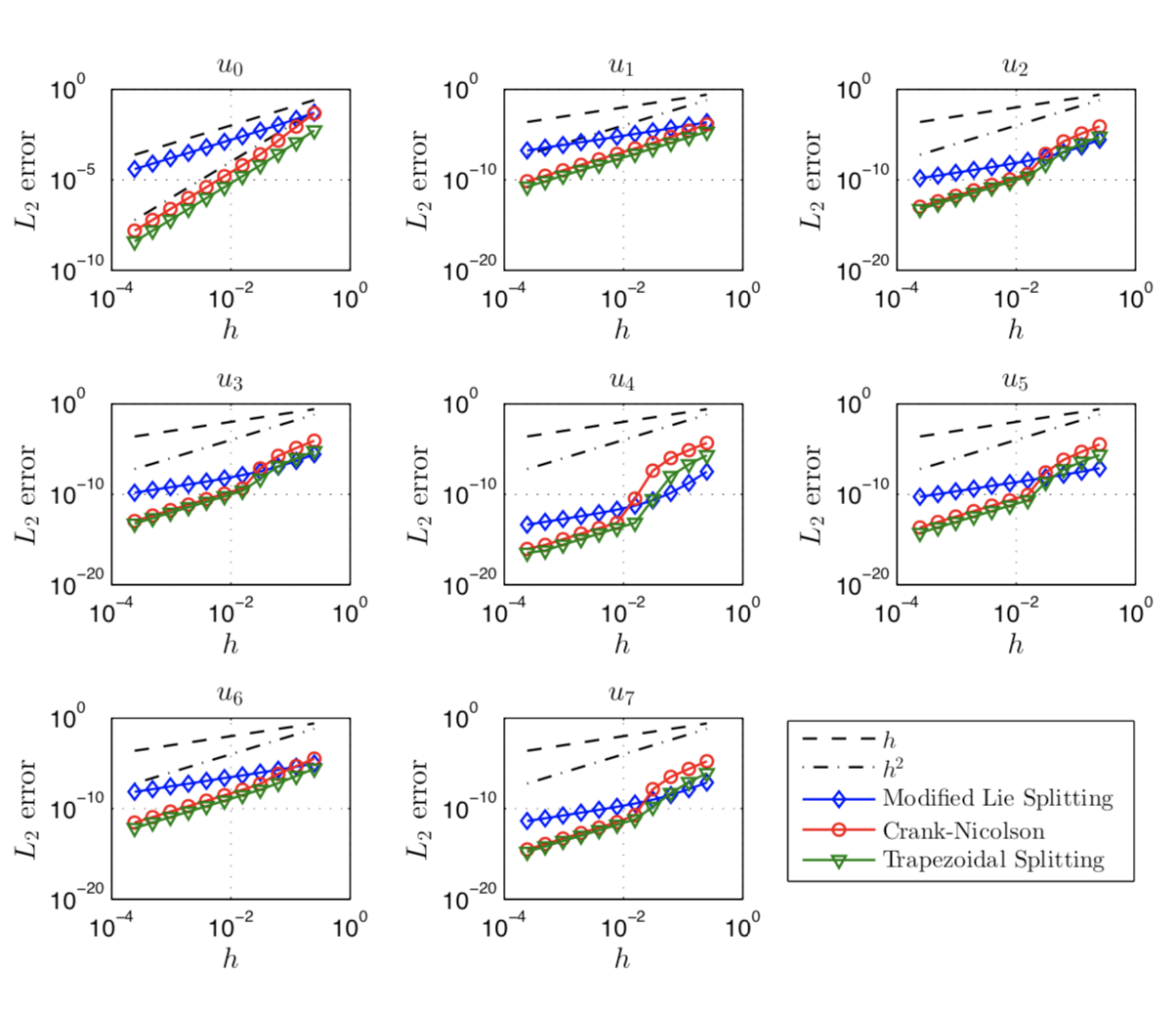}
	\caption{\small Order plots for the first eight different solutions $u_p$,  $p=0,\ldots,7$ computed with the correspondent methods.}
	\label{order_plot}
\end{figure}
Figure \ref{order_plot} shows the discrete $L^2$ error of $u_p$, $p=0,\ldots,7$ calculated with different time step sizes $h$. The time step sizes are set to $h_q = 2^{q}$ for $q=-13,\ldots,-4$.
The blue  line denotes the error of the modified Lie splitting scheme of order 1. The red line and the green line illustrate the error of the Crank--Nicolson scheme and the trapezoidal splitting method,  both of  order two. The black dashed lines have slope 1 and 2, respectively. We see that for each $m$, the order plots confirm the respective orders of the methods which can be derived from theory. \\

%
%
%
%
%

The empirical variance $\mathrm{Var}(u)$ of $u$ is given by  
\begin{equation*}
\mathrm{Var}(u) = \mathbb{E}[u-\mathbb{E}(u)] = \sum_{p=1}^P u_p^2 \, \, \mathbb{E}( \Phi_p^2),
\end{equation*}
where we used the linearity of $\mathbb{E}$ and the orthogonality of the Fourier--Legendre polynomials. Furthermore, since $u_\alpha\equiv 0$ for $|\alpha|>1$, i.e.,  $u_p \equiv 0$ for $p>m$, the number of non-zero summands in the sum is $m$ and since $\mathbb{E}(\Phi_p^2) = \frac{1}{3}$ for $1<p\leq m $,
$\mathrm{Var}(u)$ reduces to 
\begin{equation*}
\mathrm{Var}(u) = 
\frac{1}{3}\sum_{p=1}^{m} u_p^2.
\end{equation*}

\vspace{0,3cm}
Figure \ref{n_variation} shows the discrete $L^2$ error of the empirical variance of $u$ at time  $T=1$ where the summation is truncated at different $n$. The time step   $h$ used for the calculations is $h=2^{-10}$. Here, we clearly see the superiority of \\

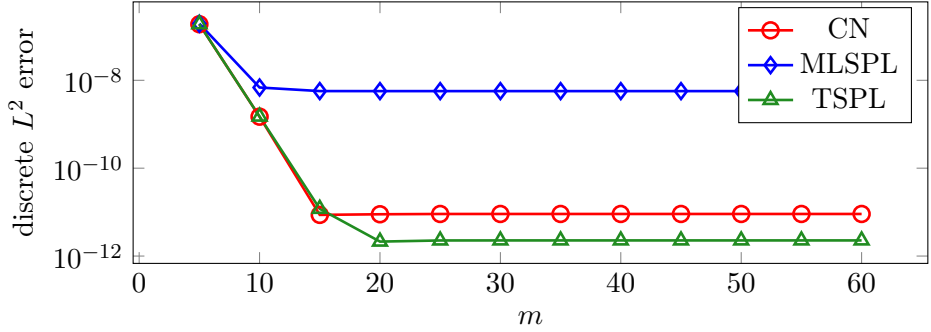
\begin{figure}[!h]
	\begin{minipage}{\linewidth}
		\begin{tikzpicture} 
		\begin{semilogyaxis}
		[xlabel=$m$,ylabel=discrete $L^2$ error] 
		\addplot[color=red,mark=o,mark size=3,line width=1] coordinates {
			(5,1.88891590837483e-07)	
			(10,1.48664964587232e-09)	
			(15,8.65292058089659e-12)	
			(20,8.93500242358620e-12)	
			(25,9.06730816627386e-12)	
			(30,9.07193666848276e-12)	
			(35,9.07227975392918e-12)	
			(40,9.07229152559456e-12)
			(45,9.07229230765565e-12)	
			(50,9.07229237245571e-12)	
			(55,9.07229237250888e-12)	
			(60,9.07229237250919e-12)
		};
		\addlegendentry{CN}
		
		\addplot[color=blue,mark=diamond,mark size=3,line width=1] coordinates { 
			(5,1.94343132529692e-07)	
			(10,6.85143022117570e-09)										
			(15,5.68078897785695e-09)	
			(20,5.67038446178132e-09)	
			(25,5.67025233359874e-09)	
			(30,5.67024773960057e-09)	
			(35,5.67024738282678e-09)	
			(40,5.67024737120269e-09)
			(45,5.67024737042091e-09)	
			(50,5.67024737035628e-09)	
			(55,5.67024737035623e-09)	
			(60,5.67024737035623e-09)
		};
		\addlegendentry{MLSPL}
		
		\addplot[color=forestgreen,mark=triangle,mark size=3,line width=1] coordinates { 
			(5,1.88898129066660e-07)
			(10,1.49164250770108e-09)	
			(15,1.18756893923273e-11)										
			(20,2.13286509587823e-12)	
			(25,2.26388587298314e-12)	
			(30,2.26850543081317e-12)	
			(35,2.26884853507600e-12)	
			(40,2.26886030876195e-12)	
			(45,2.26886109067150e-12)
			(50,2.26886115580356e-12)	
			(55,2.26886115585693e-12)	
			(60,2.26886115585723e-12)	
		};
		\addlegendentry{TSPL}						
		\end{semilogyaxis}
		\end{tikzpicture}
	\end{minipage}
	\caption{\small Discrete $L^2$ error of $\mathrm{Var}(u)$ for different number of variables $m$ used in the Karhunen--Lo\`eve expansion. The employed methods are: Crank--Nicolson (CN), modified Lie splitting (MLSPL), and trapezoidal splitting (TSPL).}
	\label{n_variation}
\end{figure}
\noindent the methods of order two compared to the modified Lie splitting for which the numerical approximation error prevails over the error induced by the truncation of the sum.  \\


\begin{table}[!h]
	\caption{\small Average computational time (in seconds) for the calculation of one solution $u_m$ for different degrees of freedom $N$, i.e., the number of computational points used in the discretization of ${\rm D}$ and the operator $\mathcal{L}$. The employed methods are: Crank--Nicolson (CN), modified Lie splitting (MLSPL) and trapezoidal splitting (TSPL).}
	\label{degrees_of_freedom}
	\begin{center}
		\begin{tabular}{c|ccc}
			\toprule
			$ N \times N$ & \textbf{CN} $[\mathrm{s}]$ & \textbf{MLSPL} $[\mathrm{s}]$ & \textbf{TSPL} $[\mathrm{s}]$\\
			\midrule 
			$4 \times 4$ & 0.0133  &  0.0373  &  0.0530\\
			$8\times 8$  &  0.0214   & 0.0241  &  0.0379\\
			$16\times 16$  & 0.1005  &  0.1008 &   0.0948\\
			$32\times 32$  & 0.4098   & 0.3652  &  0.4047\\
			$64\times 64$  & 2.7620   & 1.8051   & 1.8237\\
			$128\times 128$ &  41.1284 &   9.3921 &  13.5091\\
			\midrule
		\end{tabular}
	\end{center}
\end{table}

\vspace{0,3cm}
Finally, we report the computational work  which is needed to solve the system of partial differential equations given in \eqref{duzina alpha 0}-\eqref{duzina alpha>1}. Table \ref{degrees_of_freedom}  summarizes the computational time needed to obtain one solution of the system of partial differential equations as a function of the number of spatial grid points  $N\times N = 2^k \times 2^k$ for $k=2,3,\ldots 7$.
The highest number of grid points we are able to use (16 384) is quite low due to the fact that the calculation of the eigenvalues and eigenfunctions of the integral equation given in \eqref{cov KLE} requires the storage of a dense matrix of the size $N^2\times N^2$. 
We clearly see that the Crank--Nicolson method is by far the slowest. Both splitting methods perform approximately the for smaller $N$, while for $N = 2^7$,   Lie splitting starts to clearly outperform  trapezoidal splitting in terms of computational time.

\section{Acknowledgements}
	This work was partially supported by a  Research grant for Austrian graduates granted by the Office of the Vice Rector for Research of   University of Innsbruck.  
	The computational results presented have been partially achieved using the HPC infrastructure LEO of the University of  \linebreak  Innsbruck. A. Kofler was supported by the program Nachwuchsf\"orderung 2014 at   University of Innsbruck. H. Mena was supported by the Austrian Science Fund – project id: P27926.


\end{document}